%% file: main.tex
\documentclass[final,hidelinks,onefignum,onetabnum]{siamart250211}

\input{ex_shared}

\ifpdf
\hypersetup{
  pdftitle={Well-posedness and finite element approximation of the electrostatic shear Alfvén wave equations},
  pdfauthor={A. Buffa, T. Miehling, M. Picasso and M. Renoldner and P. Ricci}
}
\fi

\begin{document}


\maketitle

\begin{abstract}
The aim of this paper is to study the well-posedness and finite element approximation of the electrostatic shear Alfvén wave equations, a coupled system of two partial differential equations arising in plasma physics as a simplified sub-model of the drift-reduced Braginskii equations. To this end, anisotropic Sobolev spaces depending on the normalized magnetic field $\b$ are introduced, together with a Poincaré-type inequality along the integral curves of $\b$, which holds under a geometric directedness condition on the magnetic field. Using these tools, existence, uniqueness, and stability of a weak solution are established via the Faedo-Galerkin method. It is also shown that the geometric condition is satisfied in tokamak and stellarator configurations. A numerical scheme is then proposed, combining Lagrange finite elements in space with a Crank-Nicolson discretization in time. The scheme is shown to conserve a discrete energy exactly in the homogeneous case, and a priori error estimates are derived in the natural energy norm. Several numerical experiments are reported in two and three space dimensions, which confirm the theoretical results and indicate that the geometric condition on the magnetic field is necessary for the invertibility of the discrete system matrix.
\end{abstract}





\section{Introduction}
In the last decades, research in magnetic confinement fusion plasma physics has become more and more reliant on numerical analysis and computational mathematics \cite{fasoli_computational_2016, ricci_simulation_2015, schedl_holistic_2020}. Of particular interest is the theoretical understanding and development of solution algeorithms for the drift reduced Braginskii equations, a two-fluid turbulence model which describes plasma behaviour in the boundary region of fusion devices. A prominent example is the Global Braginskii Solver (GBS) a three-dimensional, finite-difference based simulation code 
\cite{ricci_GBS_2012, giacomin_gbs_2022,halpern-braginskii-cons,paruta_simulation_2018}. 
The code solves for eight unknown quantities. 
Six of them describe the plasma particle dynamics, namely the particle density, vorticity, as well as the electron and ion velocities, and temperatures, respectively. The remaining two unknowns are the electrostatic potential, and the fluctuation of the magnetic field.
The Braginskii fluid equations \cite{braginskii_transport_1965} are derived from a kinetic model, after which a so-called drift-reduction yields the drift-reduced Braginskii equations \cite{Zeiler1997Nonlinear,braginskii_transport_1965}. 


Electrostatic shear Alfvén waves \cite{jolliet_2014}, described by a coupled system of two partial differential equations, capture one of the fastest oscillatory dynamics within the drift-reduced Braginskii equations. Their numerical solution has so far been addressed using finite difference methods \cite{bassanini_2024}.

In the following, we aim to study the weak formulation of the electrostatic shear Alfvén wave equations in a general setting and to propose a robust discretisation method. To improve readability, we refer to the electrostatic shear Alfvén wave equations simply as the Alfvén waves or the Alfvén wave equations, throughout the rest of the article. 

The paper is structured as follows. In section \ref{sec:alfvenwaveequations}, we present the Alfvén wave equations and their connection to the drift-reduced Braginskii equations. In order to make sense of a weak form, and establish its well-posedness in section \ref{sec:wellposedness}, we introduce two somewhat nonstandard objects: a class of anisotropic Sobolev spaces, as well as a novel curvilinear Poincaré-type inequality along vector fields. A more detailed discussion thereof is moved to appendix \ref{sec:spaces}.
In section \ref{sec:physicaleq} we discuss in particular the example of Alfvén wave equations in the setting of tokamak and stellarator fusion experiments, which was a main motivation for this work. In section \ref{sec:numerics} we derive a finite element discretisation, and prove a priori error estimates. The discrete formulation is verified based on several numerical examples in section \ref{sec:numexamples}.




\input{Alfven_waves_on_the_Tokamak}

\input{Well-posedness}

\input{numeric_analysis}

\input{Numerical_Examples}


\section{Conclusion}
\label{sec:conclusions}
The Alfvén wave equations model the fastest oscillatory dynamics in drift-reduced Braginskii equations. Assuming a geometric property of the magnetic field, well-posedness of a weak formulation is proved. The argument relied on introducing a class of anisotropic Sobolev spaces and a novel, curvilinear Poincaré-type inequality along the magnetic field. A numerical scheme preserving energy is presented, error estimates are proved, and confirmed by numerical examples. Numerical computations also indicate that the geometric assumptions on the magnetic field, which permit to prove the Poincaré-type inequality, are needed to ensure an invertible matrix at each time step.


\appendix
\input{appendix}


\section*{Acknowledgments}

The authors are especially grateful to Paolo Ricci for securing the project's funding and for his invaluable guidance throughout this work. Further, the authors would like to thank Pascal Azerad, Micol Bassanini, and Joachim Krieger for fruitful discussions and their valuable suggestions, which helped to improve the quality of the paper.

\bibliographystyle{siamplain}
\bibliography{references}
\end{document}

%% file: ex_shared.tex
\usepackage{lipsum}
\usepackage{amsfonts}
\usepackage{graphicx}
\usepackage{epstopdf}
\usepackage{algorithmic}
\ifpdf
  \DeclareGraphicsExtensions{.eps,.pdf,.png,.jpg}
\else
  \DeclareGraphicsExtensions{.eps}
\fi

\newsiamremark{remark}{Remark}
\newsiamremark{hypothesis}{Hypothesis}
\crefname{hypothesis}{Hypothesis}{Hypotheses}
\newsiamthm{claim}{Claim}
\newsiamremark{fact}{Fact}
\crefname{fact}{Fact}{Facts}

\title{Well-posedness and finite element approximation of the electrostatic shear Alfvén wave equations\thanks{Submitted to the editors \today.
\funding{
This work has been carried out within the framework of the EUROfusion Consortium, via the Euratom Research and Training Programme (Grant Agreement No 101052200 — EUROfusion) and funded by the Swiss State Secretariat for Education, Research and Innovation (SERI). Views and opinions expressed are however those of the author(s) only and do not necessarily reflect those of the European Union, the European Commission, or SERI. Neither the European Union nor the European Commission nor SERI can be held responsible for them.
}}}

\author{Annalisa Buffa\thanks{Institute of Mathematics, EPFL, Lausanne Switzerland
  (
  \email{annalisa.buffa@epfl.ch},
  \email{timon.miehling@epfl.ch},
  \email{marco.picasso@epfl.ch},
  \email{markus.renoldner@epfl.ch}, \url{https://www.epfl.ch/schools/sb/research/math/}).}
\and Timon Miehling\footnotemark[2]
\and Marco Picasso\footnotemark[2]
\and \hspace{2cm}Markus Renoldner\footnotemark[2]
}

\usepackage{amsopn}

\usepackage{comment}

\newcommand{\RR}{\mathbb{R}}
\newcommand{\ZZ}{\mathbb{Z}}

\renewcommand{\div}{\operatorname{div}}

\renewcommand{\b}{\mathbf{b}}
\newcommand{\B}{\mathbf{B}}
\newcommand{\e}{\mathbf{e}}
\newcommand{\V}{\mathbf{V}}

\newcommand{\CP}{C_P}
\newcommand{\R}{\mathbb{R}}
\newcommand{\N}{\mathbb{N}}
\DeclareMathOperator{\Lip}{\operatorname{Lip}}

\DeclareMathOperator{\Diff}{D}
\DeclareMathOperator{\fdot}{\; \cdot \;}

\DeclareMathOperator{\dt}{d\textit{t}}
\DeclareMathOperator{\ds}{d\textit{s}}

\newcommand{\Hzero}{H_{0}^1(\Omega)}

\newcommand{\HL}{H^1(0,T;L^2(\Omega))}
\newcommand{\Ltwo}{L^2(\Omega)}
\DeclareMathOperator{\nablaperp}{\nabla_{\perp}}

\newcommand{\inner}[2]{\left\langle #1, #2 \right\rangle}

\newcommand{\innerHzero}[2]{\left\langle #1, #2 \right\rangle_{\Hzero}}
\newcommand{\norm}[1]{\left\| #1 \right\|}
\newcommand{\normL}[1]{\left\| #1 \right\|_{\Ltwo}}

\newcommand{\normHzero}[1]{\left\| #1 \right\|_{\Hzero}}

\newcommand{\normHL}[1]{\left\| #1 \right\|_{\HL}}

%% file: Alfven_waves_on_the_Tokamak.tex
\section{The Alfvén wave equations}
\label{sec:alfvenwaveequations}
Let $n \geq 2$ and let $\Omega\subseteq\RR^n$ be an open and bounded set with Lipschitz boundary. Furthermore, let $U \subseteq \R^n$ be an open subset such that $\overline{\Omega}\subset U$ and let $\B:U \to \R^n$ denote the magnetic field, which satisfies $\left\lVert \B \right\rVert\neq 0$ and $\operatorname{div}\B=0$. Finally, denote by $\b$ the normalised magnetic field, i.e., $\b:=\frac{\B}{\left\lVert \B \right\rVert}$, for which we assume $\b\in C^1(U,\R^n)$. We note that certain geometric assumptions on $\b$ have to be imposed in order to establish well-posedness; this will be discussed later. In the following we introduce the strong form of the Alfvén wave equations \cite{stasiewicz_2000,jolliet_2014,bassanini_2024}. We seek functions $u,\phi:\overline{\Omega} \times [0,T] \to \mathbb{R}$, that satisfy
\begin{alignat}{2}
    \label{eq:alfven-strong1}
    \mu \, \partial_t u - \b \cdot\nabla \phi &= f\qquad && \text{in } \Omega\times (0,T),\\
    \label{eq:alfven-strong2}
    \partial_t \operatorname{div} (\nablaperp \phi) + \operatorname{div}(\b u) &= g\qquad && \text{in } \Omega\times (0,T), \\
    \label{eq:alfven-strongBC}
    \phi&=0         \qquad && \text{in }\partial\Omega\times [0,T],\\
    \label{eq:alfven-strongIC1}
    u(\cdot,0) &= u_0        \qquad && \text{in } \overline{\Omega},\\
    \label{eq:alfven-strongIC2}
    \phi(\cdot,0) &= \phi_0  \qquad && \text{in } \overline{\Omega},
\end{alignat}
where $\mu>0$ is a given, fixed parameter describing the ratio between electron and ion mass, and $f,g$ are given functions on $\Omega \times (0,T)$ and the \emph{perpendicular gradient} operator $\nabla_\perp$ is defined as 
\begin{equation}
    \label{eq:perpgrad}
    \nabla_\perp\phi 
    :=\nabla\phi - \b (\b \cdot\nabla\phi)
    =(\mathrm{Id}-\b\b^\top)\nabla\phi
\end{equation}
and describes the projection of the gradient to the orthogonal complement of $\b$. The function $u$ describes the component of the electron velocity parallel to $\b $, $\phi$ describes the electrostatic potential, and $T>0$ denotes the final time. With $u_0,\phi_0:{\Omega} \to \mathbb{R}$ we denote the initial conditions.


We would like to remark that in fusion plasma physics, the domain of Alfvén waves are fusion reactors, such as \emph{tokamaks} or \emph{stellarators}, as explained in \cite{romanelli_collisional_transport_tokamak, Reimerdes2022TCV, Wesson2004Tokamaks}. Both are toroidal confinement devices, in which the magnetic field has a dominant component in the \emph{toroidal} direction, meaning that 
\begin{equation}
    \label{eq:global_z}
    \b(x) \cdot \e_\theta(x)> 0,
\end{equation}
where $\{\e_\rho,\e_\theta,\e_z\}$ is the cylindrical basis. It turns out that this condition is sufficient to establish well-posedness on these domains. We will first analyse the system \eqref{eq:alfven-strong1}-\eqref{eq:alfven-strongIC2} on the general domain $\Omega$, introduced in the beginning of this section, and then show explicitly that the theory applies in particular to tokamaks and stellarators, see section \ref{sec:physicaleq}.

\subsection{Alfvén waves in the Braginskii equations}
The Alfvén wave equations, as formulated in equations \eqref{eq:alfven-strong1}-\eqref{eq:alfven-strongIC2}, were first introduced as a simplified sub-model of the drift-reduced Braginskii equations in \cite{stasiewicz_2000,jolliet_2014,bassanini_2024}. This model emerges as a limit case of the drift-reduced Braginskii framework under the assumptions of cold plasma, dominant ion mass compared to the electron mass, static magnetic field, and constant plasma density. In particular, one can identify the four terms from equations \eqref{eq:alfven-strong1}-\eqref{eq:alfven-strongIC2} in equations (18) and (19) of \cite{halpern_gbs_2016}. 

A common simplification in the physics literature for the drift-reduced Braginskii equations is to assume $\nabla \cdot \mathbf{b} = 0$, which leads to terms of the form $\mathbf{b} \cdot \nabla u$ rather than $\nabla \cdot (\mathbf{b} u)$ in equation \eqref{eq:alfven-strong2}. Here, we avoid imposing this assumption.

\subsection{An a priori bound}
\label{sec:AprioriBound}
One key observation for the well-posedness of the Alfvén waves is the following conservation law when $f=g=0$. Multiplying \eqref{eq:alfven-strong1} with $u$ and \eqref{eq:alfven-strongIC2} with $\phi$ and integrating over $\Omega$ yields
$$\frac{\mathrm{d}}{\mathrm{d}t} \left ( \mu \Vert u\Vert_{L^2(\Omega)}^2 + \Vert\nabla_\perp \phi\Vert_{L^2(\Omega)}^2\right) = 0.$$
From this equation, the first natural question would be whether it is possible to control $\phi$ only through its perpendicular gradient via a Poincaré-type inequality. Such an inequality is introduced and proved in appendix \ref{sec:spaces}.
In order to prove existence of solutions, this stability is not enough, however. Indeed, controlling merely $\nabla_\perp\phi$ is not enough to meaningfully define a weak form of the second term in equation \eqref{eq:alfven-strong1} and the second term in equation \eqref{eq:alfven-strong2}. One needs either $\phi$ to admit derivatives not only perpendicular, but also along $\b$, or alternatively, $\div (\b u)$ to be integrable. The following formal argument shows, that the first option is more natural. Applying $\b\cdot\nabla$ to equation \eqref{eq:alfven-strong1} and $\partial_t$ to equation \eqref{eq:alfven-strong2}, and eliminating $u$ yields
\begin{align}
\label{eq:coupledStrong}
    \mu \partial_{tt}\div\nabla_\perp\phi + \div [\b(\b\cdot\nabla \phi)]= \mu \partial_t g - \div (\b f).
\end{align}
This equation will be utilized to derive energy estimates for the term $\b\cdot\nabla\phi$, yielding full $H^1(\Omega)$ regularity for that function.

%% file: Well-posedness.tex
\section{Well-posedness}  
\label{sec:wellposedness}
In this section, we propose a weak formulation of the Alfvén wave equations \eqref{eq:alfven-strong1}-\eqref{eq:alfven-strongIC2} and prove its well-posedness. To make sense of this weak formulation, we first introduce suitable function spaces. Equations \eqref{eq:alfven-strong1}-\eqref{eq:alfven-strongIC2} suggest that we should seek solutions in Sobolev spaces whose weak differentiability properties depend on the vector field $\b$. We then present the weak formulation of the Alfvén wave problem and state the main well-posedness result. Finally, we show that this result applies to the tokamak and stellarator cases.

\subsection{Well-posedness of the shear Alfvén waves} 
\label{sec:well-posedness}
The following defines the special vector-field-dependent Sobolev spaces required for the weak formulation of the Alfvén wave equations. For a more general presentation of these spaces, we refer to \hyperref[sec:spaces]{Appendix \ref{sec:spaces}}.
\begin{definition}[The spaces $H^1_{\b}(\Omega)$ and $H^1_\perp(\Omega)$]
    \label{def:perpgrad}
    Let $\Omega \subseteq \R^n$ be an open set and let $\b \in \Lip(\Omega)^n$ be a Lipschitz vector field. We define the spaces 
    \begin{align*}
        H^1_{\b}(\Omega) &:= \left \{ v \in L^2(\Omega) \mid \b \cdot \nabla v \in L^2(\Omega) \right \}, \\
        H^{1}_\perp(\Omega) &:= \left \{ v \in L^2(\Omega) \mid \nablaperp v \in L^2(\Omega)^n \right \}.
    \end{align*}
    Further, we define $H^{1}_{\perp,0}(\Omega)$ as the closure of $\mathcal{D}(\Omega)$ in $H^{1}_\perp(\Omega)$ with the natural choice of norm.
\end{definition}
\begin{remark}
    In case $\b$ can be completed to an orthonormal Lipschitz frame $\b_1,\ldots, \b_n \in \Lip(\Omega)^n$ with $\b_1 := \b$, we have that for each $v \in H^{1}_\perp(\Omega)$ there exist functions $f_i \in L^2(\Omega)$ for $i=2,\ldots,n$ such that we can express the perpendicular gradient as $\nablaperp v = \sum_{i=2}^n f_i \b_i$.
\end{remark}
Assuming $\mu = 1$ for simplicity, we multiply equation \eqref{eq:alfven-strong1} by a test function $v \in L^2(\Omega)$ and equation \eqref{eq:alfven-strong2} by $\psi \in H^1_0(\Omega)$. Integrating over $\Omega$ and performing integration by parts in the second equation leads to the following weak formulation.
\begin{definition}[Weak formulation of the Alfvén wave equation]
    \label{def:alfven-wave_weak}
    Let $n \geq 2$, $\Omega \subseteq \R^n$ be a bounded domain with Lipschitz boundary, $\b \in \Lip(\Omega)^n$ and let $T>0$ denote the final time. We seek $u \in H^1(0, T; L^2(\Omega))$ and $\phi \in L^2(0, T; H^1_0(\Omega)) \cap H^1(0, T; H^1_{\perp,0}(\Omega))$ such that for all $v \in L^2(\Omega)$, $\psi \in H^1_0(\Omega)$, and a.e. $t \in [0,T]$, it holds that
    \begin{align}
        \label{eq:alfven-weak1}
        \int_\Omega u' v - \int_\Omega v \, \b \cdot\nabla \phi   &= \int_\Omega f \, v, \\
        \label{eq:alfven-weak2}
        -\int_\Omega \nabla_\perp \phi'\cdot \nabla_\perp \psi  - \int_\Omega u \, \b \cdot\nabla \psi &= \int_\Omega g \, \psi,
    \end{align}
    and $u(0) = u_0$, $\phi(0) = \phi_0$, where we assume that $f \in L^2(0, T; L^2(\Omega))$,\\ $g \in L^2(0, T; H^{-1}(\Omega))$, $u_0 \in L^2(\Omega)$, $\phi_0 \in H^1_0(\Omega)$. We interpret the integral on the right hand side of \eqref{eq:alfven-weak2} as a duality-pairing.
\end{definition}
The following Poincaré-type inequality is a key ingredient for the well-posedness theorem below, and requires a specific geometric condition on $\b$.
\begin{corollary}[Poincaré-type inequality for $H_{\perp,0}^1$]
\label{CorollaryPoincare}
    Let $n \geq 2$, let $\Omega \subseteq \R^n$ be a bounded domain with Lipschitz boundary, and let $\b \in \Lip(\Omega)^n$. Suppose that there exists an orthonormal frame $\b_1,\ldots,\b_n \in \Lip(\Omega)^n$ such that $\b_1=\b$ and further such that the last frame vector $\b_n \in C^1(\R^n;\R^n)$ has bounded differential and is \emph{globally directed}, i.e., there exist $w \in \R^n$ and $\alpha > 0$ such that
    \begin{align}\label{eq:GloballyDirected}
        \b_n(x) \cdot w \geq \alpha > 0.
    \end{align}
    Then there exists $\CP>0$ depending on $\Omega$, $\b_n$ and $\alpha$ such that
    \begin{align}
    \label{eq:poincare_H1perp}
        \|v\|_{L^2(\Omega)} \leq \CP \| \nablaperp v\|_{L^2(\Omega)} \text{  for all  } v \in H^1_{\perp,0}(\Omega).
    \end{align}
\end{corollary}
\begin{proof}
    This is a direct consequence of \hyperref[thm:PoincareW1pC0]{Theorem \ref{thm:PoincareW1pC0}}.
\end{proof}
\begin{remark}
    In the following, we write $\|\fdot\|_{H^1_0(\Omega)}$ to denote $\|\nabla\fdot\|_{L^2(\Omega)}$, and $\|\fdot\|_{H^1_{\perp,0}(\Omega)}$ to denote $\|\nablaperp\fdot\|_{L^2(\Omega)}$. Both expressions define norms that are equivalent to the standard $H^1(\Omega)$ and $H^1_\perp(\Omega)$ norms on their respective subspaces, provided a Poincaré-type inequality holds.
\end{remark}
In the subsequent section \ref{sec:physicaleq} we show that these conditions are satisfied in the tokamak case and that the main result applies to this case.
\begin{theorem}[Well-posedness of the Alfvén wave equation]
    \label{thm:well-posedness}
    Under the hypotheses of \hyperref[CorollaryPoincare]{Corollary \ref{CorollaryPoincare}}, and assuming additionally that $u_0 \in H^1(\Omega)$, $g \in H^1(0,T;L^2(\Omega))$, and $f \in H^1(0,T;L^2(\Omega))$ or $f \in L^2(0,T;H^1_{\b}(\Omega))$, the problem in Definition \ref{def:alfven-wave_weak} is well-posed, that is, there exists a unique solution.
    \begin{align*}
        (u,\phi) \in H^1(0,T;L^2(\Omega)) \times \left ( L^2(0,T;H^1_0(\Omega))\cap H^1(0,T;H^1_{\perp,0}(\Omega)) \right )
    \end{align*}
    such that for all $v\in L^2(\Omega)$, $\psi\in H^1_0(\Omega)$ and a.e. $t \in [0,T]$ the equations \eqref{eq:alfven-weak1}-\eqref{eq:alfven-weak2} hold, and, moreover, there exists a constant $C_0>0$ depending only on $\Omega$, $\b$, $\alpha > 0$ and $T > 0$ such that we have
    \begin{align}
        \label{eq:AWStab}
        \left\|u\right\|_{H^1\left(0, T ; L^2(\Omega)\right)}+ &\left\|\phi\right\|_{L^2\left(0, T ; H_0^1(\Omega)\right)}  
        +\left\|\phi^{\prime}\right\|_{L^2\left(0, T ; H_{\perp, 0}^1(\Omega)\right)} \notag \\
        & \leq C_0 \left(\left\|u_0\right\|_{H^1(\Omega)}
        +\left\|\phi_0\right\|_{H_0^1(\Omega)}
        +\|f\|_X
        +\|g\|_{H^1\left(0, T ; L^2(\Omega)\right)}\right),
    \end{align}
    where $X$ is the space of $H^1(0,T;L^2(\Omega))$ or $L^2(0,T;H^1_{\b}(\Omega))$ depending on the assumption on $f$.
\end{theorem}
Well-posedness is established in several steps. We begin by proving existence, which is the most involved part. To this end, we eliminate $u$ and derive a decoupled equation for $\phi$ as an analogue of \eqref{eq:coupledStrong}. We show that this decoupled equation admits a unique solution in finite dimensions and then recover a solution to the original problem in finite dimensions from the solution of the decoupled problem. Next, we use a Galerkin-type argument to obtain the existence of a solution. Uniqueness and stability are proved subsequently. 

To define a sequence of appropriate finite-dimensional spaces, we use the eigenfunctions of the operator $-\Delta$. These are denoted by $(w_k)_{k \in \mathbb{N}}$ and form an orthonormal basis of $L^2(\Omega)$ and an orthogonal basis of $H^1_0(\Omega)$ under the given assumptions on $\Omega \subseteq \R^n$. We denote by $W_N := \operatorname{span}(w_k)_{k=1}^N$ the subspace spanned by the first $N \in \mathbb{N}$ eigenfunctions. The finite-dimensional decoupled Alfvén wave equation is an analogue of equation \eqref{eq:coupledStrong} and is given as follows.
\begin{definition}[Decoupled finite dimensional]
    \label{def:decoupled_alfven_wave}
     The finite-dimensional decoupled Alfvén wave equation is given by, find $\phi_N \in H^2(0, T; W_N)$ such that for a.e. $t \in [0,T]$ and all $\psi \in W_N$, it holds that
    \begin{equation}
        \label{eq:decoupled_findim}
        \int_{\Omega}\nabla_{\perp} \phi_N^{\prime \prime}\cdot \nabla_{\perp} \psi + \int_{\Omega}(\b\cdot\nabla\phi_N )(\b\cdot\nabla \psi) = -\int_{\Omega} f \, \b\cdot \nabla \psi - \int_{\Omega} g^{\prime} \psi,
    \end{equation}
    $\phi_N(0)=\sum_{k=1}^N w_k\int_\Omega \phi_0 w_k$ and $\phi_N'(0)$ is such that
    \begin{equation}
        \label{eq:decoupled_findim_IC}
        -\int_{\Omega} \nabla_{\perp} \phi_N^{\prime}(0)\cdot \nabla_{\perp} \psi-\int_{\Omega}u_0 \, \b\cdot\nabla \psi =\int_{\Omega} g(0) \, \psi \quad \forall \psi \in W_N.
    \end{equation}
\end{definition}
We remark, that the initial condition $\phi_N'(0)$ has to satisfy \eqref{eq:decoupled_findim_IC} precisely such that the right solution $\phi_N$ is selected, which later solves the \emph{full} Alfven wave system.
\begin{lemma}\label{lem:decoupled_existence}
    There exists a unique solution $\phi_N\in H^2 (0, T ; W_N )$ to the problem from Definition \ref{def:decoupled_alfven_wave}.
\end{lemma}
\begin{proof}
    We write the finite dimensional problem as a linear differential system of dimension $N$. For this we introduce the vector $U_N \in \R^N$, the matrices $A_N, B_N \in \RR^{N\times N}$ and the maps $F_N, G_N \in H^1(0,T)^N$ whose components are defined by
    \[
    \begin{array}{c}
        (U_N)_i := \int_{\Omega} u_{0} \, \b\cdot \nabla w_{i}, \quad
        (G_N)_i := \int_{\Omega} g \, w_i, \quad
        (F_N)_i := \int_{\Omega} f \, \b\cdot \nabla w_{i}, \\[1.5ex]
        (A_N)_{ij} := \int_{\Omega} \nabla_{\perp} w_{i} \cdot \nabla_{\perp} w_{j}, \quad
        (B_N)_{ij} := \int_{\Omega} (\b\cdot \nabla w_{i})(\b\cdot \nabla w_{j}).
    \end{array}
    \]
    The equation becomes then to find $d_N \in H^2(0,T)^N$ such that for a.e. $t \in [0, T]$ we have
    $$A_N d''_N + B_N d_N = -F_N - G'_N,$$
    $\left(d_N\right)_j(0)=\int_{\Omega}\phi_0 w_j$ for all $j=1,\ldots,N$ and $-A_N d_N^{\prime}(0)=G_N(0)+U_N.$
    Note, that due to the Poincaré-type inequality, equation \eqref{eq:poincare_H1perp},
    $A_N$ s.p.d. uniformly in $N$. Indeed, taking any $c\in\RR^N$ and $v=\sum_{i=1}^N c_i w_i$, we have that
    $$
    c ^\top A_N c
    =\int_\Omega \nabla_{\perp} v\cdot \nabla_{\perp} v 
    = \lVert \nabla_{\perp} v \rVert_{L^2(\Omega)} ^2  
    \geq \frac{1}{\CP^2} \lVert v \rVert_{L^2(\Omega)}^2 = \frac{1}{\CP^2} \|c\|^2.
    $$
    We can therefore rewrite the above system as a first order linear system and apply Carathéodory's existence theorem, see e.g. Theorem 3.4 in \cite{oregan_1997}, to infer that there exists a unique solution to the decoupled finite dimensional problem.
\end{proof}
Before constructing a finite dimensional solution pair $(\phi_N,u_N)$ to the Alfvén wave equations, we show a uniform energy estimate for $\phi_N$.
\begin{lemma}[Uniform stability of the decoupled problem in finite dimensions]
    \label{lem:decoupled_stability}
    There exists a constant $C_0>0$, depending only on  $T>0$, $\Omega \subseteq \R^n$ and $\b:\R^n \to \R^n$ such that for a.e. $t \in [0,T]$ and all $N \in \N$ it holds that 
    $$
    \left\|\phi_N\right\|_{H_0^1(\Omega)} +\left\|\phi_N^{\prime}\right\|_{H_{\perp, 0}^1(\Omega)} 
    \leq 
    C_0 \left( \left\|u_0\right\|_{H^1_{b}(\Omega)} +\left\|\phi_0\right\|_{H_0^1(\Omega)} +\|f\|_{X} +\|g\|_{H^1(0,T;L^2(\Omega))} \right),
    $$
    where $X$ is the space of $H^1(0,T;L^2(\Omega))$ or $L^2(0,T;H^1_{\b}(\Omega))$ depending on the assumption on $f$.
\end{lemma}
\begin{proof}
    Take equation \eqref{eq:decoupled_findim}, and choose $\psi=\phi_N'$. One obtains for a.e. $t\in [0,T]$ 
    \begin{equation}
        \label{eq:decoupled_findim_norm}
        \frac{\text{d}}{\text{dt}} \left [\frac{1}{2} \lVert \nabla_{\perp}\phi_{N}' \rVert_{L^2(\Omega)} ^2 + \frac{1}{2} \lVert \b\cdot \nabla \phi_{N} \rVert_{L^2(\Omega)} ^2 \right ]
        = -\int_{\Omega}f\b\cdot \nabla \phi_{N}' - \int_{\Omega}g'\phi'_{N}.
    \end{equation}
    To bound $\phi_N$, we proceed differently depending on the assumption on $f$.
    
        \emph{First case} If $f \in H^1(0,T;L^2(\Omega))$, we have that 
        $$
        \int_{\Omega} f\b\cdot \nabla \phi'_{N} = \frac{\text{d}}{\text{d}t}  \int_{\Omega} f\b\cdot \nabla \phi_{N} - \int_{\Omega} f' \b\cdot \nabla \phi_{N}.
        $$
        applying this to equation \eqref{eq:decoupled_findim_norm} and using that $\left|\int_{\Omega} f\b\cdot \nabla \phi_{N}\right| \leq \frac{1}{4}\left\|\b\cdot \nabla \phi_{N}\right\|_{L^2(\Omega)}^2+\|f\|_{L^2(\Omega)}^2$, we get that
        \begin{align*}
        \frac{\text{d}}{\text{dt}} \Big [ \frac{1}{2} \lVert \nabla_{\perp}\phi_{N}' \rVert_{L^2(\Omega)} ^2 & + \frac{1}{2} \lVert \b\cdot \nabla \phi_{N} \rVert_{L^2(\Omega)} ^2 +\int_{\Omega} f\b\cdot \nabla \phi_{N} \Big ] \\
        &\leq \frac{1}{2} \lVert \nabla_{\perp}\phi'_{N}  \rVert_{L^2(\Omega)}^2 + \frac{1}{2}\lVert \b\cdot \nabla \phi_{N} \rVert_{L^2(\Omega)}^2  + \int_{\Omega} f\b\cdot \nabla \phi_{N} \\
        &\quad + \lVert f \rVert_{L^2(\Omega)}^2 + \lVert f' \rVert_{L^2(\Omega)}^2 + \frac{\CP^2}{2} \lVert g' \rVert_{L^2(\Omega)}^2,
        \end{align*}
        where $\CP>0$ is the Poincaré constant from equation \eqref{eq:poincare_H1perp}.
        Applying Grönwall's lemma to the above yields
        \begin{align}
        \nonumber
        \frac{1}{2}&\left\|\nabla_{\perp} \phi_N^{\prime}\right\|_{L^2(\Omega)}^2 +\frac{1}{2}\left\|\b\cdot\nabla\phi_N \right\|_{L^2(\Omega)}^2+\int_{\Omega} f\b\cdot\nabla\phi_N  \\
        \leq
        \nonumber
        & \exp (t)\Bigg[ \frac{1}{2}\left\|\nabla_{\perp} \phi_N^{\prime}(0)\right\|_{L^2(\Omega)}^2+\frac{1}{2}\left\|\b\cdot\nabla\phi_N(0) \right\|_{L^2(\Omega)}^2+\int_{\Omega} f(0)\b\cdot\nabla\phi_N(0)  \\
        \label{eq:decoupled_findim_positiveterm}
        & \qquad \qquad +\int_0^t\left(\|f\|_{L^2(\Omega)}^2+ \left\| f^{\prime}\right\|_{L^2(\Omega)}^2+ \frac{\CP^2}{2} \left\|g^{\prime}\right\|_{L^2(\Omega)}^2\right)\Bigg]
        \end{align}
        We still need to bound the initial values. We begin with the term $\nabla_\perp\phi_N'(0)$. By equation \eqref{eq:decoupled_findim_IC} with $\psi:=\phi_{N}'(0)$. Elementary estimates yield
        \begin{align*}
            \lVert \phi_N'(0) \rVert_{H^1_{\perp,0}(\Omega)} ^2 \leq
            3 \CP^2 \left [ C_0^2 \lVert g \rVert_{H^1(0,T;L^2(\Omega))}^2 + \lVert \operatorname{div}\b \rVert_{L^\infty(\Omega)}^2 \lVert u_{0} \rVert_{L^2(\Omega)}^2 + \lVert \nabla u_{0}  \rVert_{L^2(\Omega)}^2 \right ],
        \end{align*}
         where we used the continuous embedding of $H^1(0,T;L^2(\Omega)) \hookrightarrow C([0,T];L^2(\Omega))$. Further, we have that
        \begin{align*}
            \normL{\b \cdot \nabla \phi_N(0)}^2 \leq \normHzero{\phi_N(0)}^2 = \sum_{k=1}^N \frac{|\innerHzero{\phi_0}{w_k}|^2}{\normHzero{w_k}^2} \leq \normHzero{\phi_0}^2
        \end{align*}
        and also
        \begin{align*}
            \left | \int_\Omega f(0) \b \cdot \nabla \phi_N(0) \right | \leq \frac{C_0^2}{2} \norm{f}_{H^1(0,T;L^2(\Omega))}^2 + \frac{1}{2} \normHzero{\phi_0}^2.
        \end{align*}
        Finally, observing that $\left | \int_\Omega f \b \cdot \nabla \phi_N \right | \leq \normL{f}^2 + \frac{1}{4} \normL{\b \cdot \nabla \phi_N}^2$ and putting everything together we get
        \begin{align*}
            \frac{1}{2} &\normL{\nablaperp \phi_N'}^2 + \frac{1}{4} \normL{\b \cdot \nabla \phi_N}^2 - \normL{f}^2 \\
            &\leq \frac{1}{2} \normL{\nablaperp \phi_N'}^2 + \frac{1}{2} \normL{\b \cdot \nabla \phi_N}^2 + \int_\Omega f \b \cdot \nabla \phi_N \\
            &\leq \exp(t) \Big [\frac{3}{2} \CP^2 \left [ C_0 \lVert g \rVert_{H^1(0,T;L^2(\Omega))}^2 + \lVert \operatorname{div}\b \rVert_{L^\infty(\Omega)}^2 \lVert u_{0} \rVert_{L^2(\Omega)}^2 + \lVert u_{0}  \rVert_{\Hzero}^2 \right ] \\
            &\quad \quad \quad \quad + \normHzero{\phi_0}^2 + \frac{C_0^2}{2} \normHL{f}^2 \\
            &\quad \quad \quad \quad + \int_0^t \left ( \normL{f}^2 + \normL{f'}^2 + \frac{\CP^2}{2} \normL{g'}^2 \right ) \Big ].
        \end{align*}
        By redefining the constant $C_0>0$ and using that $\nablaperp \phi_N(t) = \nablaperp \phi_N(0) + \int_0^t \nablaperp \phi_N'$ we get the desired bound in case of $X=H^1(0,T;L^2(\Omega))$.
        
        \emph{Second case} If $f \in L^2(0,T;H^1_{\b}(\Omega))$.
        First, observe that if $\b \in W^{1,\infty}(\Omega)^n, \varphi \in H^1_0(\Omega)$ and if $\psi \in H^1_{\b}(\Omega)$, where $\Omega \subseteq \R^n$ is as in Definition \ref{def:alfven-wave_weak}, we have the following integration by parts
        \begin{align*}
            \int_\Omega \psi \b \cdot \nabla \varphi = - \int_\Omega \left ( \div (\b)\psi + \b \cdot \nabla \psi \right ) \varphi.
        \end{align*}
        This can be proved by using a Meyers-Serrin type result for vector-field-dependent Sobolev spaces, established for example in \cite{franchi_meyers_serrin_1996}. Utilizing this result we get that
        \begin{align*}
        \left | \int_{\Omega} f\b\cdot \nabla \phi_{N}' \right | \leq \left ( \lVert \div \b \rVert_{L^\infty(\Omega)} \|f\|_{L^2(\Omega)} + \|\b \cdot \nabla f\|_{L^2(\Omega)}\right ) \|\phi_N'\|_{L^2(\Omega)}.
        \end{align*}
        We get the inequality
        \begin{align*}
            \frac{\text{d}}{\text{dt}} &\left [\frac{1}{2} \lVert \nabla_{\perp}\phi_{N}' \rVert_{L^2(\Omega)} ^2 + \frac{1}{2} \lVert \b\cdot \nabla \phi_{N} \rVert_{L^2(\Omega)} ^2 \right ] \\
            &\leq \frac{3}{2} \CP^2 \left ( \left ( \lVert \div \b \rVert_{L^\infty(\Omega)}^2 + 1 \right ) \|f\|_{H^1_{\b}(\Omega)}^2 + \|g'\|_{L^2(\Omega)}^2 \right ) + \frac{1}{2} \lVert \nabla_{\perp}\phi_{N}' \rVert_{L^2(\Omega)} ^2,
        \end{align*}
        where $\|f\|_{H^1_{\b}(\Omega)}^2 = \|f\|_{L^2(\Omega)}^2 + \|\b \cdot \nabla f\|_{L^2(\Omega)}^2$. Then applying Grönwall's lemma and using the bounds on the initial data from the first case we get the desired inequality for $X=L^2(0,T;H^1_{\b}(\Omega))$.
\end{proof}
Next we construct a solution $(u_N,\phi_N)$ to the finite dimensional version of Definition \ref{def:alfven-wave_weak}. For that we begin by observing that by the first equation in Definition \ref{def:alfven-wave_weak} $u$ is uniquely determined by $\phi$ and $u_0$. We define $u_N\in H^1\left(0, T ; L^2(\Omega)\right)$ as
\begin{equation}
    \label{eq:uN_solution_candidate}
u_N(t)=u_0+\int_0^t\left(f+\b\cdot\nabla \phi_N\right) \mathrm{d} s.
\end{equation}
It remains to check that our finite dimensional solution candidate  $(\phi_N,u_N) \in W_N \times H^1(0,T;L^2(\Omega))$ is indeed a solution.
\begin{lemma}[Existence of Alfvén waves in finite dimensions]
    Let $(\phi_N,u_N) \in W_N \times H^1(0,T;L^2(\Omega))$ be given by the solutions to the decoupled problem from Definition \ref{def:decoupled_alfven_wave} and from equation \eqref{eq:uN_solution_candidate}, respectively. Then for a.e. $t \in [0,T]$ and all $v \in L^2(\Omega)$, $\psi \in W_N$ the following equations hold
    \begin{align}
        \label{eq:alfven-findim-weak1}
        \int_\Omega u_N' v - \int_\Omega v\b \cdot\nabla \phi_N   &= \int_\Omega f v, \\
        \label{eq:alfven-findim-weak2}
        -\int_\Omega \nabla_\perp \phi_N'\cdot \nabla_\perp \psi  - \int_\Omega u_N \b \cdot\nabla \psi &= \int_\Omega g \psi,
    \end{align}
    and $u_N(0) = u_0$, $\phi_N(0) = \sum_{k=1}^N w_k \int_\Omega \phi_0 \, w_k$.
\end{lemma}
\begin{proof}
    By construction of $u_N$ equation \eqref{eq:alfven-findim-weak1} holds and also that $u_N(0)=u_0$. The initial condition for $\phi_N(0)$ is respected as well. Therefore, we only need to show that for a.e. $t \in [0,T]$ and all $\psi \in W_N$ we have that equation \eqref{eq:alfven-findim-weak2} holds. By construction of $u_N$, we have that for a.e. $t\in [0,T]$ and all $\psi\in W_N$ it holds that
    \begin{align*}
        \int_{\Omega} u_N^{\prime} \b\cdot \nabla \psi -\int_{\Omega}  (\b \cdot \nabla \phi_N ) (\b \cdot\nabla \psi) =\int_{\Omega} f  \b\cdot\nabla \psi
    \end{align*}
    If we add this to equation \eqref{eq:decoupled_findim}, we get that for a.e. $t \in [0,T]$ and for all $\psi \in W_N$ it holds that
    \begin{equation}
        \label{eq:alfven_exist_findim1}
        \int_{\Omega} \nabla_{\perp} \phi_N^{\prime \prime}\cdot \nabla_{\perp} \psi  + \int_{\Omega} u_N^{\prime}  \b\cdot\nabla \psi  = - \int_{\Omega} g^{\prime} \psi.
    \end{equation}
    Observe that for all $t\in [0,T]$ and all $\psi \in W_N$ it holds that
    \begin{align*}
        \int_{\Omega} u_N(t) \, \b \cdot \nabla \psi 
        &= \int_{\Omega} u_N(0) \, \b \cdot \nabla \psi 
        + \int_0^t \int_{\Omega} u_N'(s) \, \b \cdot \nabla \psi \, \mathrm{d}x \, \mathrm{d}s, \\
        \int_{\Omega} g(t) \, \psi 
        &= \int_{\Omega} g(0) \, \psi 
        + \int_0^t \int_{\Omega} g'(s) \, \psi \, \mathrm{d}x \, \mathrm{d}s, \\
        \int_{\Omega} \nabla_{\perp} \phi_N'(t) \cdot \nabla_{\perp} \psi  
        &= \int_{\Omega} \nabla_{\perp} \phi_N'(0) \cdot \nabla_{\perp} \psi  
        + \int_0^t \int_{\Omega} \nabla_{\perp} \phi_N''(s) \cdot \nabla_{\perp} \psi \, \mathrm{d}x \, \mathrm{d}s.
    \end{align*}
    For a fixed $t\in [0,T]$ and $\psi \in W_N$ we obtain
    \begin{align*}
    \int_{\Omega} \nabla_{\perp} \phi_N^{\prime}(t)\cdot \nabla_{\perp} \psi   + \int_{\Omega} u_N(t)  \b\cdot\nabla \psi   
    =  & \int_{\Omega} \nabla_{\perp} \phi_N^{\prime}(0)\cdot \nabla_{\perp} \psi  +\int_{\Omega} u_N(0)  \b\cdot\nabla \psi   \\
    & -\left[\int_{\Omega} g \psi-\int_{\Omega} g(0) \psi\right] .
    \end{align*}
    As $u_N(0)=u_0$ we get, by equation \eqref{eq:decoupled_findim_IC}, the second equation of the Alfvén waves.
\end{proof}
Given that $\phi_N$ solves the decoupled equation, it is no surprise that the previous lemma holds, namely that $\phi_N,u_N$ together solve the full system. This is because the second initial condition is chosen precisely so that the second equation holds at $t=0$.
\begin{lemma}[Uniform stability of Alfvén wave equations in finite dimension]
    \label{lem:stability}
    There exists a constant $C_0>0$, depending only on  $T>0$, $\Omega \subseteq \R^n$ and $\b:\R^n \to \R^n$ such that for a.e. $t \in [0,T]$ and all $N \in \N$ it holds that 
    \begin{align*}
        \left\|u_N\right\|_{L^2(\Omega)} + &\left\|u_N'\right\|_{L^2(\Omega)} + \left\|\phi_N\right\|_{H_0^1(\Omega)}  +\left\|\phi_N^{\prime}\right\|_{H_{\perp, 0}^1(\Omega)} \\
        & \leq C_0 \left(\left
        \|u_0\right\|_{H^1(\Omega)}
        +\left\|\phi_0\right\|_{H_0^1(\Omega)}
        +\|f\|_{H^1(0,T;L^2(\Omega))}
        +\|g\|_{H^1\left(0, T ; L^2(\Omega)\right)}\right),
    \end{align*}
    if $f \in H^1(0,T;L^2(\Omega))$
    and
    \begin{align*}
        \left\|u_N\right\|_{L^2(\Omega)} + &\left\|u_N'\right\|_{L^2(0,T;L^2(\Omega))} + \left\|\phi_N\right\|_{H_0^1(\Omega)}  +\left\|\phi_N^{\prime}\right\|_{H_{\perp, 0}^1(\Omega)} \\
        & \leq C_0 \left(\left
        \|u_0\right\|_{H^1(\Omega)}
        +\left\|\phi_0\right\|_{H_0^1(\Omega)}
        +\|f\|_{L^2(0,T;H^1_{\b}(\Omega))}
        +\|g\|_{H^1\left(0, T ; L^2(\Omega)\right)}\right),
    \end{align*}
    if $f \in L^2(0,T;H^1_{\b}(\Omega))$.
\end{lemma}
\begin{proof}    
    First, suppose that $f \in H^1(0,T;L^2(\Omega))$. We know from Lemma \ref{lem:decoupled_stability} that $\left\|\phi_N\right\|_{H_0^1(\Omega)} +\left\|\phi_N^{\prime}\right\|_{H_{\perp, 0}^1(\Omega)}$ is bounded uniformly for a.e. $t \in [0,T]$ and $N \in \N$. By definition of $u_N$, see equation \eqref{eq:uN_solution_candidate}, we know that
    \begin{align*}
        \lVert u'_{N} \rVert_{L^2(\Omega)} = \lVert f+b\cdot \nabla \phi_{N} \rVert_{L^2(\Omega)} \leq \lVert f \rVert_{L^2(\Omega)} + \lVert \nabla \phi_{N} \rVert_{L^2(\Omega)}.
    \end{align*}
    Consequentially
    \begin{align*}
    \left\|u_N\right\|_{L^2(\Omega)}^2 &+ \left\|u_N^{\prime}\right\|_{L^2(\Omega)}^2 \\
    &\leq C_0 \left(\left\|u_0\right\|_{L^2(\Omega)}^2+\|f\|_{H^1\left(0, T; L^2(\Omega)\right)}^2+\left\|\phi_N\right\|_{H_0^1(\Omega)}^2+\int_0^T\left\|\phi_N\right\|_{H_0^1(\Omega)}^2\right)
    \end{align*}
    for a.e. $t \in [0,T]$ and some constant $C_0>0$ depending only on $\Omega \subseteq \R^n$ and $T>0$. Using the bound on $\left\|\phi_N\right\|_{H_0^1(\Omega)}$ we already have and adapting the constant, we get the desired bound.
    \\
    Now, suppose that $f \in L^2(0,T;H^1_{\b}(\Omega))$. In that case one straightforwardly obtains
    \begin{align*}
        \left\|u_N\right\|_{L^2(\Omega)}^2 &+\left\|u_N^{\prime}\right\|_{L^2(0,T;L^2(\Omega))}^2 \\
        & \leq C_0 \left(\left\|u_0\right\|_{L^2(\Omega)}^2+\|f\|_{L^2(0, T ; H^1_{\b} (\Omega))}^2+\int_0^T\left\|\phi_N\right\|_{H_0^1(\Omega)}^2\right)
    \end{align*}
    for a.e. $t \in [0,T]$ and some constant $C_0>0$ depending only on $\Omega \subseteq \R^n$ and $T>0$. Note the slightly different stability on $u'_N$ in time. The rest of the prove is identical to the first part.    
\end{proof}
The above stability bound allows us to bound uniformly the solution $(u_N,\phi_N)$ in the space 
\begin{align*}
    H^1(0, T; L^2(\Omega)) \times \left ( L^2(0, T; H^1_0(\Omega))\cap H^1(0, T; H_{\perp,0}(\Omega)) \right ).
\end{align*}
By weak sequential precompactness of bounded sets in Hilbert spaces, see e.g. Theorem 3.18 in \cite{brezis_2010} we can pass to the limit on a subsequence, i.e. we have $(u,\phi)$ such that
\begin{align}
\label{eq:weak_convergence_u}
u_{N_\ell} \rightharpoonup u&\quad\text{in }H^1(0, T; L^2(\Omega)),\\
\label{eq:weak_convergence_phi}
\phi_{N_\ell} \rightharpoonup \phi&\quad\text{in }L^2(0, T; H^1_0(\Omega))\cap H^1(0, T; H_{\perp,0}(\Omega)).
\end{align}
\begin{lemma}[Existence of weak Alfvén waves]
    \label{lem:existence}
    The weak limits $u,\phi$ from \eqref{eq:weak_convergence_u} and \eqref{eq:weak_convergence_phi} solve the weak form of the Alfvén wave equations from Definition \ref{def:alfven-wave_weak}.
\end{lemma}
\begin{proof}
    We show that both equations and the initial conditions hold.

\emph{Equation \eqref{eq:alfven-weak1}}. Fix $v \in L^2\left(0, T ; L^2(\Omega)\right)$ and $N \in \mathbb{N}$. 
We get from the first equation that the solution $\phi_N$ of equations \eqref{eq:decoupled_findim} and \eqref{eq:decoupled_findim_IC} and the function $u_N$ defined in \eqref{eq:uN_solution_candidate} satisfy
\begin{align*}
    \int_0^T\int_{\Omega} u_N^{\prime} v-\int_0^T \int_{\Omega}v\b\cdot \nabla \phi_{N} =\int_0^T\int_{\Omega} f v.
\end{align*}
As this holds for any $N \in \mathbb{N}$ and in particular on the subsequences $N_{\ell}$ and as the expression on the LHS defines a linear, bounded functionals in $H^1\left(0, T ; L^2(\Omega)\right)$ and $L^2\left(0, T ; H_0^1(\Omega)\right)$, respectively, we can pass to the limit in the above equation. Then, choosing as test function $\eta v \in L^2(0,T;L^2(\Omega))$, where $\eta \in \mathcal{D}(0,T)$ and $v \in L^2(\Omega)$, we get by the fundamental lemma of calculus of variations that for all $v \in L^2(\Omega)$ and a.e. $t \in[0, T]$ equation \eqref{eq:alfven-weak1} holds.

\emph{Equation \eqref{eq:alfven-weak2}}. Recall that $(u_N,\phi_N)$ satisfies that for a.e. $t \in [0,T]$ and all $\psi_k \in W_N$ it holds that
$$-\int_\Omega \nabla_\perp \phi_N'\cdot \nabla_\perp \psi_k  - \int_\Omega u_N \b \cdot\nabla \psi_k = \int_\Omega g \psi_k.$$
Now, fix a positive integer $M\leq N$, multiply the equation by \( d_k(t)\in \mathcal{D}(0,T) \), sum over $k=1,\dots,M$, and integrate over $(0,T)$ to get
$$-\int_0^T \int_\Omega \nabla_\perp \phi_N'\cdot \nabla_\perp \psi_M  
- \int_0^T \int_\Omega u_N \b \cdot\nabla \psi_M =
\int_0^T  \int_\Omega g \psi_M,$$
where $\psi_M := \sum_{k=1}^M d_k(t) \psi_k$. Maps of this form are dense in $L^2(0,T;H^1_0(\Omega))$.
Next, we pass to the limit. For that we fix $M \in \N$, and $\psi_M \in \mathcal{D}(0,T;H^1_0(\Omega))$ as above and observe that the above equation holds for any $N=N_\ell$ such that $N_\ell \geq M$. For each fixed $\psi_M$, each term in the above equation is a linear bounded functional. By weak convergence, we can pass to the limit in the above equation. Then, using the density of the $\psi_M$ in $L^2(0, T; H^1_0(\Omega))$, we can replace \( \psi_M \) with any \( \psi \in L^2(0, T; H^1_0(\Omega)) \) in the above equation. Similar as in step 1 we get by the fundamental lemma of calculus of variations that for all $\psi\in H^1_0(\Omega)$ and a.e. $t\in (0,T)$, equation \eqref{eq:alfven-weak2} holds

\emph{Initial conditions}. First note the continuous embeddings of\\
$H^1(0, T; L^2(\Omega)) \hookrightarrow C([0, T]; L^2(\Omega))$ and 
$H^1(0, T; H_{\perp,0}^1(\Omega)) \hookrightarrow H^1(0, T; L^2(\Omega))$. From that we infer that $\delta_0:H^1(0, T; L^2(\Omega)) \to L^2(\Omega)$ given by the evaluation in time zero, is a continuous operator. We conclude that \( u_{N_\ell}(0) \rightharpoonup u(0) \) and \( \phi_{N_\ell}(0) \rightharpoonup \phi(0) \) in \( L^2(\Omega) \).
By the initial conditions in finite dimensions we have $u_{N_\ell}(0) = u_0$ and $\phi_{N_\ell}(0) = \sum_{j=1}^{N_\ell} w_j \int_\Omega \phi_0 w_j$. Thus we have $\phi_{N_\ell}(0) \to \phi_0$ in \( L^2(\Omega) \). By uniqueness of the weak limit we obtain \( u(0) = u_0 \) and \( \phi(0) = \phi_0 \).    
\end{proof}
\begin{lemma}[Uniqueness]
    \label{lem:uniqueness}
    The solutions $u,\phi$ to the weak Alfvén wave equations from Definition \ref{def:alfven-wave_weak} are unique.
\end{lemma}
\begin{proof}
    By linearity of the equation we only need to show that the unique solution with initial conditions $(0,0)$ and $f=g=0$ is the trivial solution. Choose $v=u$, and $\psi=\phi$ in the first and second equation, respectively. We obtain for a.e. $t\in [0,T]$
    \begin{align*}
        \frac{\text{d}}{\text{d}t} \left ( \left\lVert u \right\rVert_{L^2(\Omega)}^2 + \left\lVert \nabla_\perp \phi \right\rVert_{L^2(\Omega)}^2   \right )=0.
    \end{align*}
    Thus $\|u\|_{L^2(\Omega)}^2 + \left\lVert \nabla_\perp \phi \right\rVert_{L^2(\Omega)}^2$ is constant and as $\phi(0)=0$ and $u(0)=0$ we get that $u=0$ and $\phi=0$, giving the claim.
\end{proof}
Finally, the following claim provides the continuous dependency on the initial data, which ends the proof of Theorem \ref{thm:well-posedness}.
\begin{lemma}
    The stability bound given in equation \eqref{eq:AWStab} holds true.
\end{lemma}
\begin{proof}
    We already know that there exists a constant $C_0 > 0$ only depending on $\Omega, T > 0$, and $\b$ such that for all $N \in \N$,
    \begin{align*}
        \|u_{N} &\|_{H^1(0, T; L^2(\Omega))} + \|\phi_{N}' \|_{L^2(0, T; H^1_{\perp, 0}(\Omega))} + \|\phi_{N} \|_{L^2(0, T; H^1_0(\Omega))} \\
        &\leq C_0 \Big[ \|u_0\|_{H^1(\Omega)} + \|\phi_0\|_{H^1_0(\Omega)}
        + \|f\|_{H^1(0, T; L^2(\Omega))} + \|g\|_{H^1(0, T; L^2(\Omega))} \Big].
    \end{align*}
    By weak sequential lower semi-continuity of norms in normed spaces, we can pass to the limit, proving the claim.
\end{proof}

\subsection{Application to the tokamak}
\label{sec:physicaleq}

In this subsection, we show that the geometric hypotheses of \hyperref[thm:well-posedness]{Theorem \ref{thm:well-posedness}} are satisfied in the working conditions of fusion reactors, where the Braginskii model is used to simulate plasma turbulence, as remarked in the introduction. The geometry of such tokamak and stellarator reactors is described in more detail in \cite{romanelli_collisional_transport_tokamak, Reimerdes2022TCV, Wesson2004Tokamaks}. For this, let $\Omega$ be such a torus-shaped domain, let $U\subset \R^3$, open, such that $\overline{\Omega}\subset U$, and that $\b:\overline{\Omega} \to \R^3$ is defined as in equation \eqref{eq:global_z}, i.e. that $\b(x)\cdot\e_\theta(x) > 0$ for all $x \in U$.

We aim to show that $\b$ can be completed to an orthonormal Lipschitz frame $\b_1,\b_2,\b_3$ with $\b_1=\b$ on $\Omega$, and that $\b_3$ can be extended to a $C^1(\R^3;\R^3)$ vector field with bounded differential that is globally directed. Let ${\e_\rho,\e_\theta,\e_z}$ denote the vector fields defining the cylindrical basis. This setting is sufficient to construct an orthonormal $C^1(U;\R^3)$ frame $\b_1,\b_2,\b_3$. To this end, fix $x \in U$ and note that ${\e_\rho(x),\e_\theta(x),\e_z(x)}$ forms a basis of $\R^3$. Since $\b(x)\cdot\e_\theta(x)> 0$, it follows that ${\b(x),\e_\rho(x),\e_z(x)}$ is also a basis of $\R^3$. Applying the Gram–Schmidt procedure to ${\b(x),\e_\rho(x),\e_z(x)}$ yields an orthonormal basis ${\b_1(x),\b_2(x),\b_3(x)}$ of $\R^3$ with $\b_1(x)=\b(x)$. It is readily seen that this procedure preserves the smoothness of the initial basis. Consequently, ${\b_1,\b_2,\b_3}$ is an orthonormal $C^1(U;\R^3)$ frame, which in particular defines a Lipschitz frame on $\Omega$.

It remains to show that $\b_3$ can be extended to a globally directed $C^1(\R^3;\R^3)$ vector field with bounded differential. But this follows directly by observing that since on $U$ we have $\b_3(x) \cdot \e_z(x) > 0$ and $\Omega$ is compactly contained in $U$, we can choose a suitable cut-off function $\varphi$ such that
\[
\overline{\b}_3:= \b_3 \varphi + (1 - \varphi) \e_z
\]
defines a globally directed $C^1(\R^3;\R^3)$ extension of $\b_3|_{\Omega}$.

%% file: numeric_analysis.tex
\section{Numerical approximation}
\label{sec:numerics}

In this section we propose and analyze a numerical scheme for the weak Alfvén wave equations from Definition \ref{def:alfven-wave_weak} by utilising a Crank-Nicolson method for the time derivatives, and a finite element method in space. 

\begin{remark}
\label{rem:time-scale-invariance}
The factor $\mu$ in the first term of equation \eqref{eq:alfven-strong1} induces a time-scale invariance-like property. Given solutions $(u,\phi)$ to the problem with $\mu=1$, one can easily check that $(\sqrt{1/\mu}\ u,\phi)$ solve the problem for any $\mu$, if the time is also scaled by $\sqrt{1/\mu}$. 
This means, that choosing the realistic value $\mu=0.0005$ changes the amplitude of $u$ as well as the time scale, but nothing more: the scaling is \textit{independent} of $x\in \Omega$. Without loss of generality (a smaller $\mu$ and a smaller time step yield equivalent results) we present the following analysis for $\mu=1$.
\end{remark}

We consider a polyhedral mesh $\mathcal{T}_h$ of $\Omega$, i.e. a finite collection of polyhedra $K$ with disjoint interiors and diameter $h_K \leq h$, covering $\Omega$ exactly, and use it to define the Lagrange finite element space of order $r$
\begin{equation}
    \label{eq:femspace}
    X_h^{r}(\Omega) :=  \left \{  v_h\in C^0(\overline \Omega): v_h |_K \in \mathbb{P}^r,\ \forall K \in\mathcal{T}_h,
    \right \}.
\end{equation}
We denote by $X_{h,0}^{r}(\Omega):=X_{h}^{r}(\Omega)\cap H^1_0(\Omega)$. We assume $\mathcal{T}_h$ is shape regular, i.e. the element aspect ratio is bounded uniformly in $h$, as well as quasi-uniform. The spaces are $H^1(\Omega)$-conforming. Let $I_h^r:C^0(\overline\Omega) \to X^r_h(\Omega)$ be the Lagrange interpolation operator. For the time derivatives, we employ the Crank-Nicolson method. Let \( M \) denote the number of time steps, with \( n = 0, 1, \ldots, M \) and time step size \( \Delta t = T / M \).
For functions in \( X_h^r \) or \( X_{h,0}^r \), we use a subscript \( h \).
The superscript \(n\) indicates the value at time \(t_n := n \Delta t\) of the approximation of \(u(t_n)\) at that time instance.
For functions \( f \) and \( g \), the notation \( f^n \) and \( g^n \) similarly denote evaluation at \( t_n \).

The problem becomes, given $g \in H^1(0,T;L^2(\Omega))$ and $f \in H^1(0,T;L^2(\Omega))$ or $f \in L^2(0,T;H^1_{\b}(\Omega))$, given $u_0,\phi_0\in C^0(\overline{\Omega})$, and further given $u^n_h\in X^r_h(\Omega)$ and $\phi^n_h\in X^r_{h,0}(\Omega)$, we iteratively solve the following equations for $u_h^{n+1}\in X^r_h(\Omega)$ and $\phi_h^{n+1}\in X^r_{h,0}(\Omega)$ starting with the initial values $u^0_h:=I_h^ru(0)$ and $\phi^0_h:=I_h^r \phi(0)$ by requiring
\begin{equation}
\begin{split}
    \label{eq:alfven-fulldisc}
    \displaystyle \int_\Omega\frac{u_h^{n+1}-u_h^n}{\Delta t} v_h - \int_\Omega\b\cdot \nabla\left( \frac{ \phi_h^{n+1}+ \phi_h^n}{2} \right) v_h = \int_\Omega f(t^{n+\frac{1}{2}}) v_h, \\
    \displaystyle \int_\Omega \nabla_{\perp} \left(\frac{ \phi_h^{n+1}-\phi_h^n}{\Delta t}  \right)\cdot \nabla_{\perp}\psi_h + \int_\Omega \frac{ u_h^{n+1}+ u_h^n}{2} \b\cdot\nabla\psi_h = \int_\Omega g(t^{n+\frac{1}{2}})\psi_h,
\end{split}
\end{equation}
for all $v_h\in X_h^r(\Omega)$, $\psi_h\in X^{r}_{h,0}(\Omega)$. 
The scheme \eqref{eq:alfven-fulldisc} is stable and conservative.
\begin{lemma}[Discrete stability and energy conservation]   
\label{cor:timediscrete_stability}
    The fully discrete solution $(u_{h}^M,\phi_{h}^M)$ is stable in the energy norm, i.e. 
    there exists constants $C_1,C_2>0$ independent of $\Delta t$ such that
    \begin{align*}
    \left\lVert u_h^{M} \right\rVert_{L^2(\Omega)} + \left\lVert \nabla_\perp\phi_h^{M} \right\rVert_{L^2(\Omega)}
    \leq&\ C_{1} \left( \left\lVert u_h^0 \right\rVert_{L^2(\Omega)} + \left\lVert \nabla_\perp\phi_h^{0} \right\rVert_{L^2(\Omega)} \right) \\
    & + C_{2} \Delta t \sum_{n=0}^{M} \left( \lVert f(t^{n+\frac{1}{2}}) \rVert_{L^2(\Omega)}+\lVert g (t^{n+ \frac{1}{2}}) \rVert_{L^2(\Omega)} \right) 
    \end{align*}  
\end{lemma}

\begin{proof}
    Test equations \eqref{eq:alfven-fulldisc} with $(u^{n+1}+u^n)/2$ and $(\phi^{n+1}+\phi^n)/2$, respectively, and integrate. Consequently
    \begin{align*}
        &\frac{1}{2\Delta t} 
        \left ( 
        \left\lVert u_{h}^{n+1} \right\rVert_{L^2(\Omega)} ^2 
        +\left\lVert \nabla_\perp\phi_{h}^{n+1} \right\rVert_{L^2(\Omega)} ^2 
        - \left\lVert u_{h}^{n} \right\rVert_{L^2(\Omega)}^2 
        -\left\lVert\nabla_\perp \phi_{h}^{n} \right\rVert_{L^2(\Omega)}^2 \right )\\
        &\hspace{6cm}= 
        \int_\Omega f^{n+\frac{1}{2}}\frac{u_{h}^{n+1}+u_{h}^n}{2} - g^{n+\frac{1}{2}} \frac{\phi_{h}^{n+1}+\phi_{h}^n}{2}.
    \end{align*}
    For $f=g=0$, its clear that the scheme conserves the energy $\left\lVert u_{h}^{n} \right\rVert_{L^2(\Omega)}^2 
    +\left\lVert \nabla_\perp\phi_{h}^{n} \right\rVert_{L^2(\Omega)}^2$ exactly. Using the Poincaré inequality from equation \eqref{eq:poincare_H1perp}, and summing over all time steps yields the result.
\end{proof}

We derive now an a priori error convergence estimate for the scheme \eqref{eq:alfven-fulldisc}. For an optimal error estimate, we assume more regularity for $u,\phi$ than the minimal one provided by Theorem \ref{thm:well-posedness}. Here, $\Pi_h$ denotes the $L^2$-projection onto $X^r(\Omega)$. We denote the projection and total errors by
\begin{equation}
\label{eq:error_projections}
\begin{alignedat}{2}
\eta_{u}^n &:= \Pi_{h}u(t^n) - u^n_{h}, \quad
\theta_{u}^n &:= u(t^n) - \Pi_{h} u(t^n), \quad
\epsilon_{u}^n &:= u(t^n) - u^n_{h}, \\
\eta_{\phi}^n &:= \Pi_{h}\phi(t^n) - \phi^n_{h}, \quad
\theta_{\phi}^n &:= \phi(t^n) - \Pi_{h}\phi(t^n), \quad
\epsilon_{\phi}^n &:= \phi(t^n) - \phi^n_{h}.
\end{alignedat}
\end{equation}
\begin{theorem}[A priori error estimate]   
\label{thm:error}
    Suppose that the exact solutions $u,\phi$ from Theorem \ref{thm:well-posedness} are in $C^3([0,T];H^{r+1}(\Omega))$. Let $(u_h^M,\phi_h^M)$ be the solution of equations \eqref{eq:alfven-fulldisc}. Then, there exists a constant $C>0$, independent of $h$ and $\Delta t$ such that
    $$
    \lVert \epsilon_{u}^M  \rVert_{L^2(\Omega)} +  \lVert \nabla_{\perp}\epsilon_{\phi}^M  \rVert_{L^2(\Omega)}  
    \leq 
    CT(h^r+ \Delta t^2).
    $$
\end{theorem}
We remark that this estimate is suboptimal for $\epsilon_u$.
\begin{proof}

\input{estimate_proof}
Finally, since $u,\phi \in C^3([0,T];H^{r+1}(\Omega))$, standard $L^2$-projection and interpolation estimates yield $\lVert \eta_u^0 \rVert_{L^2(\Omega)} \leq C h^{r+1}$ and $\lVert \nabla_\perp \eta_\phi^0 \rVert_{L^2(\Omega)} \leq C h^r$, so these terms are absorbed into the $CT(h^r + \Delta t^2)$ bound. 
We combine the estimate for $\eta^M$ with the approximation error estimates $\lVert \theta_{u}^M  \rVert_{L^2(\Omega)} \leq C h^{r+1}$, and $\lVert \nabla_{\perp}\theta_{\phi}^M  \rVert_{L^2(\Omega)} \leq C h^{r}$ with constants depending on $\lVert u \rVert_{H^{r+1}(\Omega)}$ and $\lVert \phi \rVert_{H^{r+1}(\Omega)}$, respectively, which yields the statement.
\end{proof}

%% file: estimate_proof.tex
Subtracting equations \eqref{eq:alfven-weak1},\eqref{eq:alfven-weak2} from \eqref{eq:alfven-fulldisc} gives
\begin{align*}
&\displaystyle \int_\Omega\left(u'(t^{n+\frac{1}{2}}) - \frac{u_h^{n+1}-u_h^n}{\Delta t}\right) v_h
- \int_\Omega\b \cdot \nabla\left(\phi(t^{n+\frac{1}{2}}) - \frac{\phi_h^{n+1} + \phi_h^n}{2}\right) v_h
= 0,\\ 
&\displaystyle -\int_\Omega\nabla_\perp \left(\phi'(t^{n+\frac{1}{2}}) - \frac{\phi_h^{n+1}-\phi_h^n}{\Delta t}\right)\cdot \nabla_\perp \psi_h\\
&\hspace{5cm}- \int_\Omega\left(u(t^{n+\frac{1}{2}}) - \frac{u_h^{n+1}+u_h^n}{2}\right) \b \cdot \nabla \psi_h =0.
\end{align*}
Using Taylor expansions of the Crank-Nicolson time discretisation gives
\begin{align*}
\displaystyle  &\int_\Omega\frac{\epsilon_u^{n+1}-\epsilon_u^n}{\Delta t} v_h
- \int_\Omega \b \cdot \nabla \left( \frac{\epsilon_\phi^{n+1}+\epsilon_\phi^n}{2} \right) v_h 
= \mathcal{R}_u^n(v_h),\\
\displaystyle & -\int_\Omega\nabla_\perp \frac{\epsilon_\phi^{n+1}-\epsilon_\phi^n}{\Delta t} \cdot \nabla_\perp \psi_h
- \int_\Omega\frac{\epsilon_u^{n+1}+\epsilon_u^n}{2} \b \cdot \nabla \psi_h 
= \mathcal{R}_\phi^n(\psi_h),
\end{align*}
where the linear consistency error functionals satisfy, for a constant $C>0$ depending on $\lVert u \rVert_{C^3([0,T];H^{r+1}(\Omega))}$ and $\lVert \phi \rVert_{C^3([0,T];H^{r+1}(\Omega))}$ but independent of $h$ and $\Delta t$,
\begin{equation}
|\mathcal{R}_u^n(v_h)| \leq C \Delta t^2 \lVert v_h \rVert_{L^2(\Omega)}, \qquad
|\mathcal{R}_\phi^n(\psi_h)| \leq C \Delta t^2 \lVert \nabla_\perp \psi_h \rVert_{L^2(\Omega)}.
\end{equation}
Using equations \eqref{eq:error_projections} and the definition of $\Pi_h$, we get
\begin{align*}
    &\displaystyle \int_{\Omega} \frac{\eta_u^{n+1}-\eta_u^n}{\Delta t} v_h - \int_{\Omega} v_h \b \cdot \nabla \frac{\eta_\phi^{n+1}+\eta_\phi^n}{2} 
    = \int_{\Omega} v_h \b \cdot \nabla \frac{\theta_\phi^{n+1}+\theta_\phi^n}{2} + \mathcal{R}_u^n(v_h),\\    
    & \displaystyle -\int_{\Omega} \nabla_\perp \frac{\eta_\phi^{n+1}-\eta_\phi^n}{\Delta t} \cdot \nabla_\perp \psi_h - \int_{\Omega} \frac{\eta_u^{n+1}+\eta_u^n}{2} \b \cdot \nabla \psi_h \\
    &\displaystyle \hspace{1cm}= \int_{\Omega} \nabla_\perp \frac{\theta_\phi^{n+1}-\theta_\phi^n}{\Delta t} \cdot \nabla_\perp \psi_h + \int_{\Omega} \frac{\theta_u^{n+1}+\theta_u^n}{2} \b \cdot \nabla \psi_h + \mathcal{R}_\phi^n(\psi_h).
\end{align*}
To get control of $\eta_u,\eta_\phi$, we test with $v_h := \frac{\eta_u^{n+1} + \eta_u^n}{2}$ and $\psi_h := \frac{\eta_\phi^{n+1} + \eta_\phi^n}{2},$ and combine the previous two equations. After using the Poincaré inequality from equation \eqref{eq:poincare_H1perp}, as well as an inverse inequality, we get the following estimate.
\begin{align*}
&\frac{1}{2\Delta t} \left( \lVert \eta_{u}^{n+1} \rVert_{L^2(\Omega)}^2 + \lVert \nabla_{\perp}\eta_{\phi}^{n+1} \rVert_{L^2(\Omega)}^2 - \lVert \eta_{u}^n \rVert_{L^2(\Omega)}^2 - \lVert \nabla_{\perp}\eta_{\phi}^n \rVert_{L^2(\Omega)}^2 \right) \\
&\hspace{2.4cm}\leq \frac{\sqrt{ 2 }\max{(\alpha,\beta)}}{2} 
\Bigg(  \underbrace{ \sqrt{ \lVert \eta_{u}^{n+1} \rVert_{L^2(\Omega)}^2 +\lVert \nabla_{\perp}\eta_{\phi}^{n+1} \rVert_{L^2(\Omega)}^2 } }_{ =:a } \\
&\hspace{6cm}+\underbrace{ \sqrt{ \lVert \eta_{u}^{n} \rVert_{L^2(\Omega)}^2+\lVert \nabla_{\perp}\eta_{\phi}^{n} \rVert_{L^2(\Omega)}^2 } }_{ =:b } \Bigg). 
\end{align*}
with
\begin{align*}
\alpha &:= \left\lVert \b\cdot \nabla \left( \frac{\theta_{\phi}^{n+1}+\theta_{\phi}^n}{2} \right) \right\rVert_{L^2(\Omega)} + C\Delta t^2, \\
\beta &:= \left\lVert \nabla_{\perp}\left( \frac{\theta_{\phi}^{n+1}-\theta_{\phi}^n}{\Delta t} \right) \right\rVert_{L^2(\Omega)} + \frac{C_{inv}C_{p}}{h} \left\lVert \frac{\theta_{u}^{n+1}+\theta_{u}^n}{2} \right\rVert_{L^2(\Omega)} + C\Delta t^2.
\end{align*}
This is an inequality of type $(a+b)(a-b)\leq \sqrt{ 2 }\Delta t (a+b) \max(\alpha,\beta)$ which yields $a\leq b + \sqrt{ 2 }\Delta t \max(\alpha, \beta)$ or more explicitly 
\begin{equation}
\begin{aligned}
\label{eq:etabound1}
&\sqrt{\lVert \eta_{u}^{n+1} \rVert_{L^2(\Omega)}^2 + \lVert \nabla_{\perp}\eta_{\phi}^{n+1} \rVert_{L^2(\Omega)}^2}  \\
&\qquad\leq
\sqrt{\lVert \eta_{u}^{n} \rVert_{L^2(\Omega)}^2 + \lVert \nabla_{\perp}\eta_{\phi}^{n} \rVert_{L^2(\Omega)}^2}  
+ \sqrt{2}\,\Delta t\,\max(\alpha,\beta).
\end{aligned}
\end{equation}
What is left is to bound $\alpha$ and $\beta$. As $\Pi_h$ is bounded and linear, we get
\begin{align*}
\left\lVert \nabla_{\perp}\left( \frac{\theta_{\phi}^{n+1}-\theta_{\phi}^n}{\Delta t} \right) \right\rVert_{L^2(\Omega)} 
&= \frac{1}{\Delta t} \left\lVert \int_{t^n}^{t^{n+1}} \frac{\text{d}}{\text{d}s} \left[  \nabla_\perp\left( \phi(s)-\Pi_{h}\phi(s) \right)  \right] ds \right\rVert_{L^2(\Omega)} \\
&\leq \sup_{s\in (t^n,t^{n+1})}  \left\lVert \nabla_{\perp}[\phi'(s)-\Pi_{h}\phi'(s)]  \right\rVert_{L^2(\Omega)}\\
&\leq C h^{r} \lVert \phi \rVert_{C^1([0,T];H^{r+1}(\Omega))} .
\end{align*}
Elementary calculations show that we can bound 
\begin{align*}
\left\lVert \b \cdot \nabla \left( \frac{\theta_{\phi}^{n+1} + \theta_{\phi}^n }{2}\right) \right\rVert_{L^2(\Omega)} 
\leq &\ Ch^r,\\
\frac{C_{inv}C_{p}}{h} \left\lVert \frac{\theta_{u}^{n+1}+\theta_{u}^n }{2}  \right\rVert_{L^2(\Omega)} \leq &\ Ch^r.
\end{align*}
All constants $C$ above are independent of $\Delta t$ and $h$. Consequentially \eqref{eq:etabound1} becomes
$$\begin{aligned}
\sqrt{\lVert \eta_{u}^{n+1} \rVert_{L^2(\Omega)}^2 + \lVert \nabla_{\perp}\eta_{\phi}^{n+1} \rVert_{L^2(\Omega)}^2}
\leq
\sqrt{\lVert \eta_{u}^{n} \rVert_{L^2(\Omega)}^2 + \lVert \nabla_{\perp}\eta_{\phi}^{n} \rVert_{L^2(\Omega)}^2}
+ C\Delta t (h^r + \Delta t^2)
\end{aligned}$$
for a constant $C$ independent of $h$, and $\Delta t$. Summing over $n$ (and using that there are $M$ time steps, and so $\Delta t M = T$) yields
$$
\sqrt{\lVert \eta_{u}^{M} \rVert_{L^2(\Omega)}^2 + \lVert \nabla_{\perp}\eta_{\phi}^{M} \rVert_{L^2(\Omega)}^2}
\leq
\sqrt{\lVert \eta_{u}^{0} \rVert_{L^2(\Omega)}^2 + \lVert \nabla_{\perp}\eta_{\phi}^{0} \rVert_{L^2(\Omega)}^2}
+ CT (h^r + \Delta t^2).
$$

%% file: Numerical_Examples.tex
\section{Numerical examples}   
\label{sec:numexamples}


We will perform several numerical example test cases to validate our numerical scheme. We set $\mu=1$, see Remark \ref{rem:time-scale-invariance}. In all tests below, the linear system is solved using GMRES with a relative tolerance of $10^{-8}$ and with an incomplete LU preconditioner. The matrices of the discretisation are assembled using FEniCSx, \cite{fenics-UFL, dolfinx, basix, fenicsx-arbitrary-order}, and periodic boundary conditions are enforced using multipoint constraints, see \cite{dokken_dolfinx_mpc_2024}.


\subsection{Error convergence on 2D domain with a mesh-aligned $\b$-field} 


In the following we present several example computations, that should adhere to the a priori error convergence estimates presented in Theorem \ref{thm:error}. We use continuous finite elements of degree $r$ on quadrilaterals. The error $||u-u_h||_{L^2}+||\nabla_\perp \phi-\phi_h||_{L^2}$ at final time is supposed to converge at order $\Delta t^2 +h^r$. As remarked, this estimate is suboptimal for $u$, and indeed we will see that the $L^2$ error of $u$ converges one order faster in $h$.



Let $\Omega=(0,\pi)^2$ and $\b=(1,0)$ and denote by $k_x,k_y\in \mathbb{Z}$ denote the wave numbers in the $x$ and $y$ direction, respectively. Then
\begin{equation}
\label{eq:analyticalsol}
\begin{split}
\displaystyle u &= \frac{k_{y}}{\sqrt{ \mu }}\sin (k_{y}y)\left[ 1-\cos(k_{x}x)\cos \left( \frac{k_{x}}{k_{y}\sqrt{ \mu }}t \right)  \right], \\
\displaystyle \phi &= \sin \left( \frac{k_{x}}{\sqrt{ \mu } k_{y}}t \right) \sin(xk_{x})\sin(yk_{y})
\end{split}
\end{equation}
solve \eqref{eq:alfven-strong1}-\eqref{eq:alfven-strong2} with $f=g=0$ in $\Omega$ and $\phi=0$ on $\partial \Omega$. The oscillation frequency in time, $\omega$, is fully determined by $k_x,k_y$ and we have $\omega = \frac{k_{x}}{\sqrt{ \mu } k_{y}}$. As a consequence, we have a very unusual situation for linear wave problems, where short wavelengths in $x$ direction yield a low oscillation frequency $\omega$, while short wavelengths in $y$ direction correspond to high $\omega$. We set $\mu=k_x=k_y=1$ and solve the linear system using a direct solver. The results are presented in table \ref{tab:convrates_exact_r1}-\ref{tab:convrates_exact_r3}. We measure the error in the natural energy norm at final time $T=t^M=2.1$ and denote by $N_i$ the number of degrees of freedom in spatial direction $i$.
\begin{table}[ht!]
\centering
\begin{tabular}{cccccc}
$N_x=N_y$ & $N_t$ &$\Vert u(t^M)-u_h^M\Vert_{L^2}$ & rate & $\Vert \nabla_\perp(\phi(t^M)-\phi_h^M)\Vert_{L^2}$ & rate \\
\hline
4    & 10  & $2.0198 \cdot 10^{-1}$ & - & $3.1071 \cdot 10^{-1}$ & - \\
8    & 20  & $5.1764 \cdot 10^{-2}$ & 2.3163 & $1.5414 \cdot 10^{-1}$ & 1.1926 \\
16   & 40  & $1.3022 \cdot 10^{-2}$ & 2.1699 & $7.6909 \cdot 10^{-2}$ & 1.0931 \\
32   & 80  & $3.2607 \cdot 10^{-3}$ & 2.0876 & $3.8435 \cdot 10^{-2}$ & 1.0458 \\
64   & 160 & $8.1550 \cdot 10^{-4}$ & 2.0445 & $1.9215 \cdot 10^{-2}$ & 1.0227 \\
128  & 320 & $2.0390 \cdot 10^{-4}$ & 2.0224 & $9.6071 \cdot 10^{-3}$ & 1.0113 
\end{tabular}
\caption{Convergence rates for exact solutions, $r=1$.}
\label{tab:convrates_exact_r1}
\end{table}
\begin{table}[ht!]
\centering
\begin{tabular}{cccccc}
$N_x=N_y$ & $N_t$ &$\Vert u(t^M)-u_h^M\Vert_{L^2}$ & rate & $\Vert \nabla_\perp(\phi(t^M)-\phi_h^M)\Vert_{L^2}$ & rate \\
\hline
4    & 10   & $1.3368 \cdot 10^{-2}$ & - & $3.3270 \cdot 10^{-2}$ & - \\
8    & 30   & $1.6384 \cdot 10^{-3}$ & 3.3007 & $7.8350 \cdot 10^{-3}$ & 2.2737 \\
16   & 90   & $1.7617 \cdot 10^{-4}$ & 3.3620 & $1.9495 \cdot 10^{-3}$ & 2.0972 \\
32   & 270  & $1.9934 \cdot 10^{-5}$ & 3.2144 & $4.8710 \cdot 10^{-4}$ & 2.0458 \\
64   & 810  & $2.3901 \cdot 10^{-6}$ & 3.0946 & $1.2176 \cdot 10^{-4}$ & 2.0227 
\end{tabular}
\caption{Convergence rates for exact solutions, $r=2$.}
\label{tab:convrates_exact_r2}
\end{table}
\begin{table}[ht!]
\centering
\begin{tabular}{cccccc}
$N_x=N_y$ & $N_t$ &$\Vert u(t^M)-u_h^M\Vert_{L^2}$ & rate & $\Vert \nabla_\perp(\phi(t^M)-\phi_h^M)\Vert_{L^2}$ & rate \\
\hline
4    & 10    & $1.0430 \cdot 10^{-2}$ & - & $6.3801 \cdot 10^{-3}$ & - \\
8    & 40    & $6.5427 \cdot 10^{-4}$ & 4.2343 & $4.6126 \cdot 10^{-4}$ & 4.0172 \\
16   & 160   & $4.0897 \cdot 10^{-5}$ & 4.1199 & $4.0199 \cdot 10^{-5}$ & 3.6260 \\
32   & 640   & $2.5559 \cdot 10^{-6}$ & 4.0602 & $4.3083 \cdot 10^{-6}$ & 3.2704 \\
64   & 2560  & $1.6063 \cdot 10^{-7}$ & 4.0221 & $5.1355 \cdot 10^{-7}$ & 3.0916 
\end{tabular}
\caption{Convergence rates for exact solutions, $r=3$.}
\label{tab:convrates_exact_r3}
\end{table}
The numerical results in Tables \ref{tab:convrates_exact_r1}–\ref{tab:convrates_exact_r3} are consistent with the a priori convergence estimate of Theorem \ref{thm:error}. For $r=1$, the method achieves second-order convergence in the measured error of $u$ and first order for $\nabla_\perp\phi$, as expected. For $r=2$ and $r=3$, the observed rates in both $u$ and $\phi$ match the predicted orders, showing that the proposed scheme performs well.

\subsection{Error convergence on 3D domain with non-aligned $\b$-field}
\label{sec:3Dconvtest}
We run now a slightly more complex, three-dimensional test case on a periodic domain, 
$\Omega = \left(L_{\rho,\text{min}},L_{\rho,\text{max}}\right)\times\RR /(2\pi\ZZ)\times\left(0,L_z\right)$, modeling a tokamak reactor as introduced in sections \ref{sec:alfvenwaveequations} and \ref{sec:physicaleq} and compute manufactured solutions for a magnetic field which is not aligned with the mesh. We choose $L_{\rho,\text{min}}=0.5$, $L_{\rho,\text{max}}=1$, and $L_z=1$. We remark that the following tests have been conducted using a cylindrical coordinate system, which introduces the usual correction factors depending on the radius $\rho$ in the integrals and the derivatives. We introduce
\begin{align*}
    \b_\epsilon(\rho,\theta,z)= \left(\sqrt{1-2\epsilon^2},\epsilon,\epsilon\right),
\end{align*}
and compute errors for $\epsilon=0.2$. We define the manufactured solutions
\begin{align*}
    u_{ex}(\rho,\theta,z) &:= \rho\sin(\theta),\\
    \phi_{ex}(\rho,\theta,z) &:= 
    \sin \left( \frac{\pi(\rho-L_{\rho,\text{min}})}{L_\rho-L_{\rho,\text{min}}} \right)
    \sin\left(\frac{z\pi}{L_z}\right) \sin(\theta)
\end{align*}
and use the Alfven wave equations \eqref{eq:alfven-strong1}, \eqref{eq:alfven-strong2} to compute the corresponding $f,g$. We measure the error in the natural norm at final time $T=t^M=0.1$ and denote by $N_i$ the number of degrees of freedom in spatial direction $i$. The results are presented in table \ref{tab:convrates_manuf_eps0.2_r1} and \ref{tab:convrates_manuf_eps0.2_r2}. 

\begin{table}[ht!]
\centering
\begin{tabular}{cccccc}
$N_i$ & $N_t$ &$\Vert u(t^M)-u_h^M\Vert_{L^2}$ & rate & $\Vert \nabla_\perp(\phi(t^M)-\phi_h^M)\Vert_{L^2}$ & rate \\
\hline
4    & 1    & $1.9401 \cdot 10^{-1}$ & - & $1.2373 \cdot 10^{0}$ & - \\
8    & 2    & $4.8199 \cdot 10^{-2}$ & 2.3694 & $4.9234 \cdot 10^{-1}$ & 1.5679 \\
16   & 4    & $1.2172 \cdot 10^{-2}$ & 2.1641 & $2.2216 \cdot 10^{-1}$ & 1.2514 \\
32   & 8    & $3.0538 \cdot 10^{-3}$ & 2.0848 & $1.0766 \cdot 10^{-1}$ & 1.0922 \\
64   & 16   & $7.6728 \cdot 10^{-4}$ & 2.0379 & $5.3390 \cdot 10^{-2}$ & 1.0348 \\
\end{tabular}
\caption{Convergence rates for manufactured solutions for $\b_{0.2}$, $r=1$.}
\label{tab:convrates_manuf_eps0.2_r1}
\end{table}

\begin{table}[ht!]
\centering
\begin{tabular}{cccccc}
$N_i$ & $N_t$ &$\Vert u(t^M)-u_h^M\Vert_{L^2}$ & rate & $\Vert \nabla_\perp(\phi(t^M)-\phi_h^M)\Vert_{L^2}$ & rate \\
\hline
4   & 1   & $1.9262 \cdot 10^{-2}$ & -      & $1.1854 \cdot 10^{-1}$ & -      \\
8   & 3   & $2.8444 \cdot 10^{-3}$ & 3.0079 & $2.4036 \cdot 10^{-2}$ & 2.5093 \\
16  & 9   & $5.0992 \cdot 10^{-4}$ & 2.5917 & $5.5790 \cdot 10^{-3}$ & 2.2022 \\
32  & 27  & $1.2478 \cdot 10^{-4}$ & 2.0768 & $1.4546 \cdot 10^{-3}$ & 1.9832 \\
\end{tabular}
\caption{Convergence rates for manufactured solutions for $\b_{0.2}$, $r=2$.}
\label{tab:convrates_manuf_eps0.2_r2}
\end{table}



As visible in tables \ref{tab:convrates_exact_r1} and \ref{tab:convrates_manuf_eps0.2_r1}, the convergence rates for $r=1$ are well above the prediction from Theorem \ref{thm:error}. The convergence rate of the error $\Vert u-u_h\Vert_{L^2}$ is actually a full degree higher than the theoretical estimation. Similarly, the results from table \ref{tab:convrates_exact_r2} for $r=2$ are above the threshold, while the results from table \ref{tab:convrates_manuf_eps0.2_r2} for $r=2$ appear to be slightly below the predicted rate for the error on $\Vert u-u_h\Vert_{L^2}$. The convergence rate of $\Vert \nabla_\perp(\phi-\phi_h)\Vert_{L^2}$ behaves as predicted, i.e. reaches a value slightly higher than $1$ for $r=1$, and surpasses a value of $2$ for $r=2$.

\subsection{Energy conservation}
To validate further the correctness of the algorithm, we compute a solution to the homogeneous problem, i.e. with $f=g=0$ and check in each time step, whether the discrete energy $\mathcal{E}_h(t):=\int_\Omega \left(u_h(t)^2 +|\nabla_\perp \phi_h(t)|^2 \right)$ is indeed exactly conserved, as discussed in lemma \ref{cor:timediscrete_stability}. We report the following example computation to validate this conservation result. On a mesh of size $60\times 60\times 5$ on the domain $\Omega=(0.5,1)\times(0,1)\times\RR/2 \ZZ$ and with the realistic mass ratio value $\mu=0.0005$, we solve Alfven waves using 60 time steps and a final time of $T=2$. The magnetic field $\b$ is chosen as in section \ref{sec:3Dconvtest}. Let the initial condition be 
\begin{align*}
u(r,\varphi,z) &= 
\sin\!\left( \pi \varphi \right)\cos\!\left( \pi z \right)
\;+\;
\sin\!\left( 2\pi \varphi \right)\cos\!\left(2 \pi z \right),\\
\phi(r,\varphi,z) &=  0.
\end{align*}
\begin{table}[ht!]
\centering
\begin{tabular}{cc}
$t$ & $\mathcal{E}$ \\
\hline
0& $1.4095940948\cdot 10^{-4}$\\
2& $1.4095940750\cdot 10^{-4}$ 
\end{tabular}
\end{table}
The difference between the energy at $t=0$ and $t=T=2$ is $1.976\times 10^{-12}$. A similar computation on a coarser mesh, $30\times 30\times 5$ with 300 time steps up to final time $T=20$ yields a similar result. The absolute energy difference between $t=0$ and $t=T=20$ is $2.176 \times 10^{-12}$.
\begin{table}[H]
\centering
\begin{tabular}{cc}
$t$ & $\mathcal{E}$ \\
\hline
0  & $1.3176183247 \cdot 10^{-4}$ \\
10 & $1.3176183364 \cdot 10^{-4}$ \\
20 & $1.3176183465 \cdot 10^{-4}$ 
\end{tabular}
\end{table}

\subsection{Necessity of the geometric condition on the magnetic field}
From the proof Lemma \ref{lem:decoupled_existence} it is clear that the matrix $A_N$ is invertible if the Poincare-inequality from equation \eqref{eq:poincare_H1perp} holds, which itself holds if e.g. $\b$ is globally directed in the sense of equations \eqref{eq:global_z} and \eqref{eq:GloballyDirected}. In this section, however, we would like to examine how the direction of $\b$ influences the invertibility of the system matrix.
We consider a two-dimensional periodic problem to illustrate the dependency of the invertibility of the perpendicular Laplace operator on the orientation of a magnetic field vector $\b$. The computational domain is the unit square with periodic boundary conditions imposed in the $x$-direction and homogeneous Dirichlet conditions on $y=0$ and $y=1$, $\Omega = \RR/\ZZ \times (0,1)$. The magnetic field is given by
\[
\b_\epsilon:x\mapsto \left(\sqrt{1-\epsilon^2}, \epsilon\right), \quad \epsilon \in [0,1],
\]
so that $\b_\epsilon(x)$ remains a unit vector for all $\epsilon$. Recall the definition of the $\nabla_\perp$ operator from Definition \ref{def:perpgrad} and the induced bilinearform $\int_\Omega \nabla_\perp \phi\cdot\nabla_\perp\psi$ from the first term of equation \eqref{eq:alfven-weak2}. 

\begin{table}[h]
\centering
\begin{tabular}{cccc}
$\epsilon$ & Condition number & $\lambda_{\max}$ & $\lambda_{\min}$ \\
\hline
0          & $1.20 \times 10^{2}$ & $3.9021$ & $3.2629 \times 10^{-2}$ \\
0.9        & $1.72 \times 10^{2}$ & $3.1933$ & $1.8599 \times 10^{-2}$ \\
0.99       & $1.98 \times 10^{3}$ & $3.8571$ & $1.9480 \times 10^{-3}$ \\
0.999      & $2.01 \times 10^{4}$ & $3.9269$ & $1.9568 \times 10^{-4}$ \\
0.9999     & $2.01 \times 10^{5}$ & $3.9340$ & $1.9576 \times 10^{-5}$ \\
0.99999    & $2.01 \times 10^{6}$ & $3.9347$ & $1.9577 \times 10^{-6}$ \\
0.999999   & $2.01 \times 10^{7}$ & $3.9347$ & $1.9577 \times 10^{-7}$ \\
0.9999999  & $2.01 \times 10^{8}$ & $3.9347$ & $1.9577 \times 10^{-8}$ \\
0.99999999 & $2.01 \times 10^{9}$ & $3.9347$ & $1.9577 \times 10^{-9}$ \\
1 & $1.14 \times 10^{17}$ & $3.9347$ & $-7.3237 \times 10^{-16}$
\end{tabular}
\caption{Condition number and spectrum of the perpendicular Laplace matrix for varying $\epsilon$.}
\label{tab:perp_cond}
\end{table}

The trial and test functions $\phi,\psi$ belong to the Lagrange finite element space $X^1_{h,0} \subset H^1_{\perp,0}(\Omega)$ introduced below equation \eqref{eq:femspace}. We use a $10\times 10$, fully structured quadrilateral mesh.
The condition number is computed using \texttt{numpy} which itself uses an SVD to compute the condition number. Table \ref{tab:perp_cond} summarises the numerical results. The condition number remains moderate for small $\epsilon$, but grows rapidly as $\epsilon \to 1$. At $\epsilon=1$, the matrix is effectively singular, considering that the computations were performed in double precision arithmetic.

%% file: appendix.tex
\section{Vector-field dependent Sobolev spaces and Poincaré-type inequality}
\label{sec:spaces}
To state the Alfvén waves in a suitable weak formulation, we use the framework of \cite{maione_2020,maione_2020a,maione_2022,maione_2023}, who worked on certain examples of what we call \emph{vector-field-dependent Sobolev spaces}. Some properties of these spaces, such as traces, are also discussed in \cite{di_pietro_ern_2012}. In this section we provide a short presentation of those spaces.
\begin{definition}[Vector-field-dependent Sobolev spaces]
\label{def:vec_field_dep_sob}
Let $\Omega \subseteq \R^n$ be an open set, $p \in [1,\infty]$, and let $\mathcal{C}=\{\V_1,\ldots,\V_m\}$, where $\V_1,\ldots,\V_m \in \operatorname{Lip}(\Omega)^n$ be a collection of Lipschitz vector fields. We define the \emph{vector-field-dependent Sobolev space} induced by the collection of vector fields $\mathcal{C}$, as
\begin{align*}
    W^{1,p}_{\mathcal{C}}(\Omega):= \left \{v \in L^p(\Omega) \mid \V_i \cdot \nabla v \in L^p(\Omega),\ \forall i=1,\ldots,m \right \}.
\end{align*}
Further, we define $\nabla_\mathcal{C} := [\V_1 \cdot \nabla,\dots,\V_m \cdot \nabla]^\top$ and in case of $p=2$ we write $H^{1}_{\mathcal{C}}(\Omega):=W^{1,2}_{\mathcal{C}}(\Omega)$.
\end{definition}
One can verify that if we choose $\mathcal{C}$ to be the identity matrix, then we recover the classical Sobolev spaces, i.e., $W^{1,p}_{\text{Id}_n}(\Omega) = W^{1,p}(\Omega)$. Moreover, with the natural choice of norm, these spaces are Banach spaces, and Hilbert spaces in the case $p=2$. We now define the analogue of the space $H^1_0(\Omega)$ within the framework of \emph{vector-field-dependent Sobolev spaces}.
\begin{definition}[Closure of $\mathcal{D}(\Omega)$ in $W^{1,p}_{\mathcal{C}}(\Omega)$]
    Let $\Omega \subseteq \mathbb{R}^n$ be open, $p\in [1,\infty]$, and let $\mathcal{C}=\{\V_1,\ldots,\V_m\}$ be a collection of Lipschitz vector fields. We define $W^{1,p}_{\mathcal{C},0}(\Omega)$ as the closure of $\mathcal{D}(\Omega)$ in $W^{1,p}_\mathcal{C}(\Omega)$, i.e.,
    \[
     W_{\mathcal{C},0}^{1,p}(\Omega) := \overline{\mathcal{D}(\Omega)} \quad \text{ w.r.t. } \quad W_\mathcal{C}^{1,p}(\Omega).
    \]
    In case $p=2$ we write $H^{1}_{\mathcal{C},0}(\Omega)$.
\end{definition}
As $W^{1,p}_{\mathcal{C},0}(\Omega)$ is defined as the closure of a linear subspace of $W^{1,p}_\mathcal{C}(\Omega)$, it is a closed subspace of $W^{1,p}_\mathcal{C}(\Omega)$. In particular, it is a Banach space since $W^{1,p}_\mathcal{C}(\Omega)$ is complete.


In the following, we establish the Poincaré-type inequality that is essential for the well-posedness result stated in \hyperref[thm:well-posedness]{Theorem \ref{thm:well-posedness}}. We assume a geometric condition on the vector field that allows us to \emph{straighten} it and then apply the classical one-dimensional Poincaré inequality along the resulting straightened integral curves. The following results were first presented in a slightly weaker form in \cite{curved-poincare,azerad2007}. In particular, we do not assume here that the vector field is divergence-free. Alternative proofs and variations of the inequality in different settings can be found in \cite{brunken_2019, muga_2019, besson_2007,franchi_poincare_1995}.
\begin{definition}[Globally directed vector fields]
\label{def:globally_directed}
A vector field $\V:\mathbb{R}^n\rightarrow\mathbb{R}^n$ is said to be \emph{globally directed}, if there exists, a fixed direction $w\in \mathbb{R}^n$, $\Vert w\Vert =1$, and a constant $\alpha\in \RR$ such that
\begin{align*}
    \V(x)\cdot w \geq \alpha > 0\quad \forall x \in \mathbb{R}^n.
\end{align*}
\end{definition}
Now we construct the diffeomorphism to straighten globally directed vector fields.
\begin{proposition}[Diffeomorphism for globally directed vector fields]
    \label{lem:diffeomorphism}
    Let $\V \in C^k(\mathbb{R}^n;\mathbb{R}^n)$ with $k \geq 1$ be globally directed vector field with bounded differential.  Then there exists a $C^k$-diffeomorphism \(\Phi(t,s) : \R^n = \R \times \R^{n-1} \to \R^n\) such that
    \[
        \frac{\partial}{\partial t} \Phi(t,s) = \V \left(\, \Phi(t,s)\right) \quad \forall (t,s) \in \R \times \R^{n-1}.
    \]
\end{proposition}
\begin{proof}
    The proof strategy is close to the one used in the Flowout theorem, see e.g. Theorem 9.20 in \cite{lee_2003}. As $\V$ is a $C^k$ vector field it induces a $C^k$ flow $\Theta$. This flow is global as $\Theta$ has bounded differential.
    \\
    Let \(\{w_2, \dots, w_n\}\) be an orthonormal basis for \(\operatorname{span}\{w\}^\perp\). Then \(\{\V(x), w_2, \dots, w_n\}\) is a basis for \(\R^n\) in each $x \in \R^n$. Define the map
    \begin{align*}
        &\Phi : \R^n = \R \times \R^{n-1} \to \R^n \quad \text{by} \quad (t,s) \mapsto \Theta \left ( t, \sum_{i=2}^n w_i s_i \right ) = \Theta \circ L (t,s) \quad \text{with} \\
        &L(t,s) = L \cdot [t, s^\top]^\top \quad \text{and} \quad L = 
        \begin{bmatrix}
            1 & 0 & \dots & 0 \\
            \mathbf{0} & w_2 & \dots & w_n \\
        \end{bmatrix}
        \in \R^{(n+1) \times n},
    \end{align*}
    This map is clearly $C^k(\R^n;\R^n)$ and we claim that it satisfies the claim. We prove this in a three steps. First, we show that the map is an immersion, followed by proving \(\frac{\partial}{\partial t} \Phi = \V \circ \, \Phi\) and finally, we show bijectivity of \(\Phi\).
        \\
        \\
        \textit{Invertibility of differential} : To show that the differential is invertible in each point \((t, s) \in \R \times \R^{n-1}\), we begin by computing the differential at a point \((0,s) \in \R \times \R^{n-1}\). Fix \((t, x) \in \R \times \R^n\) and \((v_t, v_x) \in \R \times \R^n\). Then 
        \[
        \Diff \Theta(t, x)[v_t, v_x] = \Diff_t \Theta(t, x)[v_t] + \Diff_x \Theta(t, x)[v_x],
        \]
        where \(\Diff_t, \Diff_x\) are the partial differentials w.r.t.\ \(t\) and \(x\). Clearly, \\
        $\Diff_t\Theta(t, x)[v_t] = \V(\Theta(t,x)) \cdot v_t$ and if \(t = 0\), we have that \(\Theta(0, \fdot) = \operatorname{Id}_{\R^{n}}\). Thus
        \begin{align*}
            \Diff \Theta(0,x) = 
            \begin{bmatrix}
                \V(x), e_1, \dots, e_n
            \end{bmatrix}
            \in \R^{n \times (n+1)}.
        \end{align*}
        By the chain rule, we get $\Diff \Phi(0, s) = 
        \begin{bmatrix}
            \V(x), w_2, \dots, w_n
        \end{bmatrix}$, where\\
        \(x = \sum_{i=2}^n s_i w_i\). This matrix is invertible, giving the claim for \((0, s) \in \R^n\).
        \\
        Now, for a general point \((t, s) \in \R \times \R^{n-1}\), we use a well-chosen translation. Fix \(t \in \R\) and define \(\tau_t : \R \times \R^n \to \R \times \R^n\) by \((t_0, x) \mapsto (t_0 + t, x)\). 
        \\
        We have that 
        \[
            \Theta \circ \tau_t(t_0, x) = \Theta(t+t_0, x) = \Theta(t,\Theta(t_0,x))=\Theta_t(\Theta(t_0,x))=\Theta_t \circ \Theta(t_0,x),
        \]
        showing that $\Theta \circ \tau_t = \Theta_t \circ \Theta$, where $\Theta_t:=\Theta(t,\fdot)$. From the fundamental theorem on flows, see e.g. Theorem 9.12 in \cite{lee_2003} we know that \(\Theta_t : \R^n \to \R^n\) is a $C^k(\R^n;\R^n)$-diffeomorphism. Therefore, taking the differential, we get
        \begin{align*}
            \Diff \Theta(t_0+t, x) = \Diff \Theta_t (\Theta(t_0,x)) \cdot \Diff \Theta(t_0,x).
        \end{align*}
        Taking \(t_0 = 0\), $(t,s) \in \R^{n}$ fixed and \(x = \sum_{i=2}^n s_i w_i\) we get that
        \[
            \Diff \Phi(t, s) = \Diff \Theta(t, x) \cdot L = \Diff \Theta_t(x) \cdot \Diff \Theta(0,x) \cdot L = \Diff \Theta_t(x) \cdot \Diff \Phi(0,s).
        \]
        As \(\Theta_t : \R^n \to \R^n\) is a diffeomorphism, we have that \(\Diff \Phi(t, s)\) is invertible.
        \\
        \\
        \textit{Derivative in $t$ direction} : This follows directly by observing that for $(t,s) \in \R \times \R^{n-1}$ we have that
        \begin{align*}
            \frac{\partial}{\partial t} \Phi(t,s) = \frac{\partial}{\partial t} (\Theta \circ L )(t,s) = \Diff \Theta(t,x) \cdot e_1 = \frac{\partial}{\partial t} \Theta(t,x) = \V \circ \, \Phi (t,s).
        \end{align*}
        \\
        \\
        \textit{Bijectivity} : We begin by showing surjectivity of \(\Phi\). Let \(x \in \R^n\). We know that \(t \mapsto \Theta(t, x)\) is the unique global integral curve of \(\V\) starting at \(x \in \R^n\). Define the map 
        \[
            W_{\Theta^{(x)}}(t) := \langle \Theta^{(x)}(t), w \rangle, \text{ where } \Theta^{(x)}(t):=\Theta(x,t)
        \]
        and observe that the map is $C^k$ with derivative $W_{\Theta^{(x)}}'(t) \geq \alpha > 0$, giving that \(W_{\Theta^{(x)}}(t)\) is strictly monotone increasing, and \(W_{\Theta^{(x)}}(t) \to \pm \infty\) as \(t \to \pm \infty\).
        \\
        By the intermediate value theorem, there exists \(t_0 \in \R\) such that \(W_{\Theta^{(x)}}(t_0) = 0\), which implies that $\Theta^{(x)}(t_0) \in \operatorname{span}\{w_2, \dots, w_n\}$. Thus there exists \(s_0 \in \R^{n-1}\) which is mapped to \(x_0 := \Theta^{(x)}(t_0)\).
        \\
        By uniqueness of the integral curves we get $\Theta(-t_0, x_0) = x$. This proves surjectivity. It remains to show injectivity.
        \\
        Let \((t_1, s_1), (t_2, s_2) \in \R \times \R^{n-1}\). Suppose for contradiction that \((t_1, s_1) \neq (t_2, s_2)\) and \(\Phi(t_1, s_1) = \Phi(t_2, s_2)\). Let \(x_1, x_2 \in \operatorname{span}\{w\}^\perp\) be the images of \(s_1, s_2\) by the map \(s \mapsto \sum_{i=2}^n s_i w_i\). If \(x_1 = x_2\), then \(t_1 \neq t_2\) and we get a contradiction as $W_{\Theta^{(x)}}(t)$ is strictly increasing.
        \\
        Suppose that \(x_1 \neq x_2\). We get that $\Theta^{(x_1)}, \Theta^{(x_2)}$ are different integral curves, which are such that there exists $t_1,t_2 \in \R$ such that \(\Theta^{(x_1)}(t_1) = \Theta^{(x_2)}(t_2)\). Both are global integral curves of \(\V\) starting at \(x_0 := \Theta^{(x_1)}(t_1) = \Theta^{(x_2)}(t_2)\). By uniqueness, we therefore have \(\Theta^{(x_1)}(t_1+t) = \Theta^{(x_2)}(t_2+t)\). If $t_1=t_2$ we get the contradiction $x_1=x_2$. If $t_1\neq t_2$, then $\Theta^{(x_2)}(0) = x_2$ and $\Theta^{(x_2)}(t_2-t_1) = x_1$, which is again a contradiction as \(W_{\Theta^{(x_2)}}\) is strictly increasing.
\end{proof}
Putting together proposition \ref{lem:diffeomorphism} and the standard Poincaré inequality we get the following generalisation of the usual Poincaré inequality to vector-field-dependent Sobolev spaces.
\begin{theorem}[Poincaré-type inequality for \(W^{1,p}_{\mathcal{C},0}(\Omega)\)-spaces]
    \label{thm:PoincareW1pC0}
    Let \(\Omega \subseteq \R^n\) be open and bounded, \(1 \leq p < \infty\) and $C:\Omega \to \mathbb{R}^{m \times n}$ be defined as in Definition \ref{def:vec_field_dep_sob}. Suppose there exists $j \in \{1,\ldots,m\}$ such that $\V_j \in C^1(\R^n;\R^n)$, has bounded differential and is globally directed as in Definition \ref{def:globally_directed}. Then there exists a constant \(\CP > 0\), only depending on \(\Omega\) and \(\mathcal{C}\), such that
    \begin{align}
        \label{eq:PoinIneq}
        \|u\|_{L^p(\Omega)} \leq \CP \|\nabla_\mathcal{C} \, u\|_{L^p(\Omega)} \text{  for all  } u \in W^{1,p}_{\mathcal{C},0}(\Omega).
    \end{align}
\end{theorem}
\begin{proof}
    By proposition \ref{lem:diffeomorphism}, there exists a $C^1$-diffeomorphism \(\Phi(t, s)\) such that
    \begin{align*}
        \frac{\partial}{\partial t} \Phi(t, s) = \V_j \circ \Phi(t, s) \quad \forall \; (t,s) \in \R \times \R^{n-1}.
    \end{align*}
    Observe that we only need to proof the inequality \eqref{eq:PoinIneq} on $\mathcal{D}(\Omega)$ and that for \(\varphi \in \mathcal{D}(\Omega)\), we have that \(\partial_{\V_j} \varphi = \langle \nabla \varphi, \V_j \rangle\).
    \\
    Now, fix \(\varphi \in \mathcal{D}(\Omega)\) and observe that, as \(\overline{\Omega} \subset \R^n\) is compact, also \(\Phi^{-1}(\overline{\Omega})\) is compact. Take a box $B^n := \prod_{i=1}^n [m_i,M_i] \subset \R^n$ with $m_i < M_i \in \R$ for all $i=1,\dots,n$ such that $\Phi^{-1}(\overline{\Omega}) \subset B^n$. Define $B^{n-1} \subseteq \R^{n-1}$ such that $B = [m_1, M_1] \times B^{n-1}$. Next, observe that as $\Phi$ is a diffeomorphism we have that $|J_{\Phi}| = |\det(\Diff \Phi)|$ is bounded on $B^n$, $|J_{\Phi^{-1}}| = |\det(\Diff \Phi^{-1})|$ is bounded on $\overline{\Omega}$ and by bijectivity of $\Phi$ we have $\overline{\Omega} \subseteq \Phi(B^n)$. Thus we get
    \[
        \|\varphi\|_{L^p(\Omega)}^p = \int_{\R^n} |\varphi|^p =  \int_{\R^n} |\varphi \circ \Phi|^p |J_\Phi| \leq \| J_\Phi \|_{L^\infty(B^n)} \int_{B^n} |\varphi \circ \Phi|^p.
    \]
    By Fubini/Tonelli $\int_{B^n} |\varphi \circ \Phi|^p = \int_{B^{n-1}} \int_{[m_1,M_1]} |\varphi \circ \Phi|^p \dt \ds$. As for all \(s \in B^{n-1}\), \(\varphi \circ \Phi(\fdot,s) \in W^{1,p}_0(m_1,M_1)\), we get that there exists \(C_1 > 0\) only dependent on \((m_1, M_1)\) such that
    \[
        \left \|\varphi \circ \Phi(\fdot,s) \right \|_{L^p(m_1,M_1)}^p \leq C_1^p \left \| \langle \nabla \varphi, \V_j \rangle \circ \Phi(\fdot, s) \right \|_{L^p(m_1,M_1)}^p
    \]
    by the standard Poincaré inequality on \(W^{1,p}_0(m_1,M_1)\). We get
    \begin{align*}
        \int_{B^{n-1}} \int_{m_1}^{M_1} |\varphi \circ \Phi|^p \dt \ds &\leq C_1^p \int_{B^n} |\inner{\nabla \varphi}{\V_j} \circ \Phi|^p \\
        &\leq C_1^p \|J_{\Phi^{-1}}\|_{L^\infty(\Omega)} \int_{\Omega} |\inner{\nabla \varphi}{\V_j}|^p
    \end{align*}
    by doing the inverse change of variables. As the constant \(\|J_{\Phi} \|_{L^\infty(B^n)} C_1^p \|J_{\Phi^{-1}}\|_{L^\infty(\Omega)}\) clearly only depends on \(\Omega\) and \(\V_j\), the claim is proved.
\end{proof}

\begin{remark}
The constant $C_P$ depends on the $t$-intervall length $\ell:=M_1 - m_1$, which itself is determined by $\Omega$ and 
the lower bound $\alpha$ in Definition \ref{def:globally_directed}; if 
$\Omega$ is \emph{large} or $\alpha \to 0$, then $\ell$ diverges, and therefore 
$C_P$ as well.
\end{remark}